\documentclass[12pt]{amsart}
\def \R{\mathbb{R}}
\def \Z{\mathbb{Z}}

\def \T{\mathscrsfs{T}} 
\def \Critb{{\mathop\mathrm{Crit}}^\bullet}
\def \Crit{{\mathop\mathrm{Crit}}}

\def\cat{{\mathop\mathrm{cat}}}

\def\co{\colon}

\def\Bcat{{\mathop\mathrm{Bcat}}}
\def\Int{{\mathop\mathrm{Int}}}
\def\n{^{-1}}
\def\bs{\backslash}
\newcommand{\pa}[2]{\dfrac{\partial #1}{\partial #2}}
\usepackage{soul}
\usepackage[mathscr]{euscript}

\DeclareSymbolFont{rsfs}{U}{rsfs}{m}{n}
\DeclareSymbolFontAlphabet{\mathscrsfs}{rsfs}

\usepackage{tikz-cd}

\usepackage{graphicx, amscd, amsmath, amsfonts, amsthm, amssymb, color, xparse, hyperref}
\usepackage{mathabx}
\newtheorem{theorem}{Theorem}
\newtheorem{lemma}[theorem]{Lemma}
\newtheorem{corollary}[theorem]{Corollary}
\newtheorem{conjecture}[theorem]{Conjecture}

\newtheorem{proposition}[theorem]{Proposition}

\theoremstyle{definition}
\newtheorem{definition}[theorem]{Definition}

\newtheorem{remark}[theorem]{Remark}
\newtheorem{example}[theorem]{Example}

\usepackage{geometry}


\title[The minimal number of critical points]{The minimal number of critical points of a smooth function on a closed manifold and the ball category.}
\author{Rustam Sadykov}
\address{Kansas State University}
\email{sadykov@ksu.edu}
\author{Stanislav Trunov}
\address{Kansas State University}
\email{stastrunov@ksu.edu}
\date{\today}

\begin{document}
\begin{abstract}   
Introduced by Seifert and Threlfall, cylindrical neighborhoods of isolated critical points of smooth functions is an essential tool in the Lusternik-Schnirelmann theory. We conjecture that every isolated critical point of a smooth function admits a cylindrical ball neighborhood. 
We show that the conjecture is true for cone-like critical points, Cornea reasonable critical points, and critical points that satisfy the Rothe $H$ hypothesis. In particular, the conjecture holds true at least for those critical points that are not infinitely degenerate.

If, contrary to the assertion of the conjecture, there are isolated critical points that do not admit cylindrical ball neighborhoods, then we say that such critical points are exotic.  We prove a Lusternik-Schnirelmann type theorem asserting that the minimal number of critical points of smooth functions without exotic critical points on a closed manifold of dimension at least $6$ is the same as the minimal number of elements in a Singhof-Takens filling of $M$ by smooth balls with corners. 

\end{abstract}
\maketitle

\section{Introduction}

Given a smooth function $f\co M\to \R$ on a manifold $M$ without boundary, a point $x$ in $M$ is said to be \emph{critical} if the differential $d_xf$ of $f$ at $x$ is trivial. A point that is not critical is said to be \emph{regular}.  
We say that a value $c$ of $f$ is \emph{critical} if the fiber $f^{-1}(c)$ contains a critical point.  Critical points of functions could be extremely complicated. For example, the fiber $f^{-1}(c)$ over a regular value $c$ is a hypersurface of $M$ with no singularities. On the other hand, every closed subset $V\subset M$ is the hypersurface level $f^{-1}(c)$ of an appropriately chosen smooth function $f$ on $M$. In other words, critical points of smooth functions are at least as complicated as closed subsets of a manifold. 

Let $x_0$ be an isolated critical point of a smooth function $f$ with critical value $f(x_0)=c$. %{\color{red} A vector field $V$ is \emph{weakly gradient like} if $V(f)> 0$ outside the set of critical points of $f$, and the set of critical points of $f$ coincides with the set of zeroes of $V$.  
We say that an isolating compact neighborhood $U$ of $x_0$  is \emph{cylindrical} if there exists $\varepsilon>0$ such that $U$ consists of trajectories in $f^{-1}[c-\varepsilon, c+\varepsilon]$ of the gradient vector field of $f$ with respect to a Riemannian metric on $M$, see \S\ref{s:pt1}. We will always assume that $U$ is a smooth manifold of dimension $\dim M$ with corners such that $U\cap f^{-1}(c-\varepsilon)$ and $U\cap f^{-1}(c+\varepsilon)$ are smooth manifolds with (smooth) boundary. Cylindrical neighborhoods were introduced by Seifert and Threlfall \cite{ST38} in 1938, and were used in \cite{Ro50, Ta68, Da84} and \cite{Co98}.  In the presence of cylindrical ball neighborhoods of critical points, i.e., cylindrical neighborhoods homeomorphic to a ball,  one may apply, for example, the Lusternik-Schnirelman type argument  and deduce Lusternik-Schnirelman inequalities, see \cite{Ta68, Co98}, and Theorem~\ref{Susan}. We note that in dimensions $\ne 4,5$ all smooth manifolds homeomorphic to a ball are diffeomorphic to a ball.

In the first part of the paper we study cylindrical ball neighborhoods and give evidence supporting Conjecture~\ref{c:main_cyl}.  

\begin{conjecture}\label{c:main_cyl}
Every isolated critical point of any smooth function admits a cylindrical ball neighborhood. 
\end{conjecture}

If there is an open subset $U\subset \R^m$, and a coordinate neighborhood $\tau\co U\to M$ of a point $x\in M$  such that $f\circ \tau$ is the restriction of a polynomial, then we say that the critical point $x$ of $f$ is \emph{algebraic}. In this case, the singular hypersurface $V=f^{-1}(c)$ is locally an algebraic set. It is known (e.g. see \cite{Mi69}) that for every algebraic critical point $x_0$ of $f$ with $f(x_0)=c$, there is a closed disc neighborhood $B_\varepsilon$ of $x_0$ in $M$ such that $\partial B_\varepsilon$ is transverse to $V$, and  the pair $(B_\varepsilon, V\cap B_\varepsilon)$ is homeomorphic to the cone over the pair $(\partial B_\varepsilon, V\cap \partial B_\varepsilon)$. We say that $(B_\varepsilon, V\cap  B_\varepsilon)$ is a \emph{cone neighborhood pair} for $x_0$.  
 This essential property of algebraic singular points is a starting point for the theory of hypersurface singularities, see \cite{Mi69}.   
It follows (see \cite{Ki78, Ki90}, and Proposition~\ref{prop:cone16}) that a critical point $x_0$ admits a cone neighborhood pair if and only if it is \emph{cone-like}, i.e., 
if it admits a cone neighborhood in $f^{-1}(f(x_0))$.

\begin{theorem}\label{th:1} Every cone-like critical point of any smooth function admits a cylindrical ball neighborhood. 
\end{theorem}

A version of Theorem~\ref{th:1} appears in \cite[p. 396]{Ki78} for manifolds $M$ of dimension $m\ne 4$.

\iffalse
\sout{Every cylindrical neighborhood is an adapted neighborhood \cite{Fu21}, i.e., it is a closed isolated connected neighborhood $U$ of a critical point of a continuous function such that the level set $f^{-1}(a)$ is transverse to $\partial U$ for each $a$ in the interior of $f(U)$. In} \cite{Ki78, Ki90}  \sout{King showed that each cone-like critical point of a continuous function on a manifold of dimension $m \ne 4$ admits an adapted ball neighborhood} \cite[Lemma 2.1]{Fu21}.
\sout{Theorem~\ref{th:1} recovers the result of King in the case of smooth functions, and covers the missing case $m=4$. }

{\color{blue} I decided to remove the previous paragraph. In his papers King is very informal. In the second paper on page 3 he says that "by the argument in" the first paper, he can show the existence of a nice cylindrical neighborhood. In the first paper, he gives a 3 line sketch that such a neighborhood is a ball when $m\ne 4$   }
\fi

In \cite{Co98} Cornea defined a \emph{reasonable critical point} as an isolated critical point $x_0$ of $f$ such that $f^{-1}(f(x_0))$ admits a Whitney stratification into the stratum $\{x_0\}$ and its complement. By \cite[Lemma 2]{Co98} every reasonable critical point admits a cone neighborhood pair, and, therefore,  it
is cone-like. 

\begin{corollary}\label{c:reas} Every reasonable critical point admits a cylindrical ball neighborhood.
\end{corollary}

Another large class of critical points was studied by E.~Rothe in \cite{Ro50} and \cite{Ro52}. We say that an isolated critical point $x_0$ satisfies the hypothesis $H$ if there is an isolating coordinate neighborhood $U$ of $x_0$ such that for all $x\ne x_0$ in $U\cap f^{-1}(x_0)$ the vectors $x-x_0$ and $\nabla f(x)$ are linearly independent. 

\begin{corollary}\label{cor_R} Every critical point satisfying the hypothesis $H$ admits a cylindrical ball neighborhood. 
\end{corollary}

We will show that the set of jets of non cone-like critical points is of infinite codimension  in the space of all jets, see Corollary~\ref{c:t17}. In other words, non cone-like critical points are extremely rare and infinitely degenerate.  We also deduce that the critical point of any infinitely determined map germ is cone-like, see Proposition~\ref{p:inf20}. 

Finally, we show that the class of those critical points that admit cylindrical neighborhoods diffeomorphic to a ball is larger than the class of cone-like critical points. 

\begin{theorem}\label{th:Tak} For every $m\ge 4$ there is a smooth function $f\co \R^m\to \R$ with an isolated critical point $x_0$ such that $x_0$ is not cone-like and such that $x_0$ admits a cylindrical  neighborhood diffeomorphic to a ball. 
\end{theorem}

The functions $f$ that appear in Theorem~\ref{th:Tak} are constructed by Takens in \cite{Tak}. 
The argument in the proof of Theorem~\ref{th:Tak} 
implies that critical points of Takens functions are removable, see Corollary~\ref{c:Tak26}. This 
is closely related to the Funar theorem~\cite[Proposition 2.1]{Fu21} asserting that every cone-like critical point of smooth maps $f\co M\to N$ of a manifold of dimension $m\le 2k-1$ to a manifold of dimension $k\ge 2$ are topologically removable except possibly when $(m, k)\in \{(2,2), (4, 3), (8,5), (16, 9)\}$, and under an additional condition that the critical point admits an adapted neighborhood diffeomorphic to a ball when $(m,k)=(5,3)$.

Theorem~\ref{th:1} and its corollaries are related to the Lusternik-Schnirelmann theory \cite{Co98, Si79, Ta68}, which we discuss in the second part of the paper.

The Lusternik-Schnirelmann category $\cat(X)$ of a topological space $X$ is the least integer $n$ such that the space $X$ admits a covering by $n+1$ open subspaces, each of which is contractible in $X$ to a point. By the Lusternik-Schnirelmann theorem~\cite{LS34}, when $X$ is a closed manifold, the numeric invariant $\cat(X)$ gives a lower bound for the number $\Crit(f)$ of critical points of any smooth function $f$ on $X$, namely, 

\begin{equation}\label{eq:1}
   \cat(X) + 1 \le \Crit (X),
\end{equation}
where $\Crit(X)$ is the minimum of $\Crit(f)$ for all smooth functions $f$ on $X$.
We note that the differential geometry invariant $\Crit(X)$ is hard to compute in general, while the numeric invariant $\cat(X)$ is a homotopy invariant, which, at least in some cases, can be computed by means of homotopy theoretic methods, e.g., see \cite{CLOT}.

We will establish a Lusternik-Schnirelmann type equality for a numeric invariant associated with Singhof-Takens fillings. A Singhof-Takens filling~\cite{Ta68, Si79} of a closed smooth manifold $M$ of dimension $m$ is essentially a covering of $M$ by compact smooth submanifolds of dimension $m$. The compact submanifolds may have corners, while the interiors of covering submanifolds are required to be disjoint. Of particular interest are Singhof-Takens fillings by contractible manifolds, smooth balls with corners, and topological balls which are smooth manifolds homeomorphic to balls, see Remark~\ref{rem:CWhi}. 
The least number $n$ such that every Singhof-Takens filling of $M$ by smooth (respectively,  topological) balls has at least $n+1$ elements is denoted by $\Bcat (M)$ (respectively, $\Bcat^{\T}(M$)). If $M$ is a compact smooth manifold with boundary then $\Bcat(M)$ (respectively, $\Bcat^{\T}(M)$) is the least number $n$ such that every relative filling (see Definition~\ref{RelFilling}) of $M$ by smooth (respectively,  topological) balls with corners has at least $n+1$ elements. We also require the function $f$ minimizing the number of critical points on $M$ in the definition of $\Crit (M)$ to be regular and constant maximal over $\partial M$.
We denote by $\Critb(M)$ the minimum of numbers $\Crit(f)$, where $f$ ranges over all functions on $M$ that only have critical points admitting cylindrical ball neighborhoods.

We will show that $\Bcat(M)$ and $\Bcat^\T (M)$ are closely related to $\Critb(M)$.  

\begin{theorem} \label{Billy} Let $M$ be a smooth compact manifold of dimension at least $6$. Then $\Bcat(M)+1 =\Critb (M)$.  
\end{theorem}

The statement of Theorem~\ref{Billy} is true for manifolds $M$ with $\dim (M)=2$. In the case where $\dim (M)=3$ the conclusion of Theorem~\ref{Billy} also remains true;  it is essentially the Takens theorem (see the proof of \cite[Theorem 3.3]{Ta68}) together with Proposition~\ref{p:Tak19}.

%Since clearly $\cat(M)\le \Bcat^\T(M)$, Theorem~\ref{Billy} implies the Lusternik-Schnirelmann estimate. 

Since every smooth manifold of dimension $\ne 4,5$ homeomorphic to a ball is actually diffeomorphic to a ball, the two invariants $\Bcat^\T(M)$ and $\Bcat(M)$ agree for manifolds $M$ of dimension $\ne 4,5$. In particular, Theorem~\ref{Billy} follows from general estimates of Theorem~\ref{Susan} that hold for manifolds of arbitrary dimension. 

\begin{theorem}\label{Susan} Let $M$ be a smooth compact manifold. Then $\Bcat^{\T}(M)+1 \le \Critb(M) \le \Bcat(M) +1$. 
\end{theorem}

Since $\Crit(M)\le \Critb(M)$, Theorem~\ref{Susan} establishes an upper bound for $\Crit(M)$. 

\begin{corollary} Let $M$ be a smooth compact manifold. Then $\Crit(M)\le \Bcat(M)+1$.
\end{corollary}

Theorem~\ref{Susan} recovers the Cornea theorem~\cite{Co98}, in which a manifold $M$ is essentially decomposed into cones rather than balls, while functions  
$f$ are restricted to have only \emph{reasonable critical points}, i.e., isolated critical points $x_0$ such that $f^{-1}(f(x_0))$ is Whitney stratified into two strata, namely, $\{x_0\}$ and its complement. 
The minimum number of critical points for such functions on $M$ is denoted by $\Crit^{W}(M)$.

\begin{theorem}[Cornea, 1997]\label{th:C} Let $M$ be a smooth compact manifold. Then $\mathop\mathrm{cl}(M)+1\le \Crit^{W}(M)$, where $\mathop\mathrm{cl}(M)$ is the cone length of $M$. 
\end{theorem}

 %For manifolds of arbitrary dimension we have a general estimate. 

%Theorem~\ref{Billy} follows from Theorem~\ref{Susan} as every smooth manifold of dimension at least 6 homeomorphic to a ball is actually diffeomorphic to a ball, and therefore $\Bcat^\T(M)=\Bcat(M)$  for manifolds $M$ of dimension at least $6$. 

%Every reasonable critical point $x_0$ of a smooth function $f$ admits \cite{Co98} a closed neighborhood $V$ in $f^{-1}(f(x_0))$ such that $V$ is a cone over $\partial V$. Consequently,

 By Corollary~\ref{c:reas}, every reasonable critical point $x_0$ admits a cylindrical ball neighborhood. On the other hand, every topological ball is a cone. Consequently,  
\[
   \mathop\mathrm{cl}(M)+1\le \Bcat^\T(M)+1\le \Critb(M)\le \Crit^W(M),
\]
and therefore, Theorem~\ref{Susan}  indeed implies Theorem~\ref{th:C}.

\iffalse
\begin{remark} 
An isolated singular point $x_0$ of a smooth map $M\to N$ is \emph{cone-like} if it admits a cone neighborhood in $f^{-1}(f(x_0))$. Every cone-like critical point of a function admits a cylindrical ball neighborhood, but the converse is not true \cite{ST}.  It is known that if $f$ has a finitely determined map germ at $x_0$, then $x_0$ is cone-like. On the other hand, the set of jets of map germs that are not finitely determined is a proalgebraic subset of infinite codimension in the space of all jets. In other words, unless $f$ is infinitely degenerate at $x_0$, the point $x_0$ is cone-like, and therefore admits a cylindrical ball neighborhood (for details, see \cite{ST}). 
\end{remark}
\fi

\begin{conjecture}\label{c:6}
All critical points of smooth functions admit cylindrical ball neighborhoods, and therefore $\Bcat(M)+1=\Crit M$ when $\dim M\ne 4, 5$. 
\end{conjecture}

Conjecture~\ref{c:6} is closely related to the Funar Conjecture \cite{Fu21}. Namely,  the minimal number of critical points of smooth maps $M\to N$ with only cone-like critical points is denoted by $\varphi_c(M, N)$, while the minimal number of critical points of functions $M\to N$ is denoted by $\varphi(M, N)$. Funar conjectured \cite{Fu21} that $\varphi(M, N)=\varphi_c(M,N)$, and proved the conjecture under the assumption  $\dim M\le 2\dim N-1$.

\iffalse 
It is known that if $f$ has a finitely determined map germ at $x_0$, then $x_0$ admits a cone neighborhood in $f^{-1}(f(x_0))$. On the other hand, the set of jets of map germs that are not finitely determined is a proalgebraic subset of infinite codimension in the space of all jets. In other words, unless $f$ is infinitely degenerate at $x_0$, the point $x_0$ admits a cone neighborhood. Takens constructed an example of a function $f$ with a unique critical point $x_0$ at which all derivatives of $f$ vanish. The critical point $x_0$ constructed by Takens does not admit a cone neighborhood in $f^{-1}(f(x_0))$ but it still admits a cylindrical ball neighborhood \cite{ST}. We conjecture that all critical points admit cylindrical ball neighborhood, and therefore $\Bcat(M)+1=\Crit(M)$ when $\dim M\ge 6$. This conjecture is closely related to the Funar Conjecture \cite{Fu21}. Namely, let $f$ be a smooth map $M\to N$ with finitely many singular points. A singular point is \emph{cone-like} if it admits a cone neighborhood in $f^{-1}(f(x_0))$. The minimal number of critical points of functions $M\to N$ with only cone-like critical points is denoted by $\varphi_c(M, N)$, while the minimal number of critical points of functions $M\to N$ is denoted by $\varphi(M, N)$. Funar conjectured \cite{Fu21} that $\varphi(M, N)=\varphi_c(M,N)$, and proved the conjecture under the assumption $\dim M\le 2\dim N-1$. 

\fi

\subsection*{Organization of the paper}
In section~\ref{s:pt1} we give a detailed proof of Theorem~\ref{th:1}. Next, in section~\ref{s:3C4} we prove Corollary~\ref{cor_R}, and show that every cone-like critical point admits a cone neighborhood pair. In section~\ref{s:4alg} we review relevant theorems from singularity theory and deduce that the set of jets of map germs with non cone-like critical points is a subset of a proalgebraic set of infinite codimension in the space of all jets. We also show that critical points of infinitely determined map germs are cone-like. Finally, in section~\ref{s:5Tak}, we list properties of Takens critical points, namely: the Takens function germs are flat (Proposition~\ref{p:23Tak}), the critical points of Takens functions admit cylindrical neighborhoods (Proposition~\ref{prop:21}) diffeomorphic to a ball, and the critical points of Takens functions are removable (Corollary~\ref{c:Tak26}). The proof of Theorem~\ref{th:Tak} can also be found in section~\ref{s:5Tak}.

In section~\ref{s:stf2} we review Singhof-Takens theory and introduce relative fillings. 
 In section~\ref{s:ls5} we establish the lower bound for $\Crit^\bullet(M)$ of Theorem~\ref{Susan}, see Theorem~\ref{th:13lower}. 
The upper bound for $\Crit^\bullet(M)$ of Theorem~\ref{Susan} is proved in section~\ref{s:5}, see Theorem~\ref{th:18upper}, by an argument that relies on the Takens theorem (see Theorem~\ref{th:13}).

\vskip 3mm
We are grateful to Louis Funar for valuable comments, and references.

\section{Proof of Theorem~\ref{th:1}}\label{s:pt1}
 
 Let $M$ be a smooth, compact or open, connected manifold, and let $f:M\to \R$ be a smooth proper function with at most finitely many critical points. If $M$ is compact with non-empty boundary, then we assume that $f$ is regular over $\partial M$, and $f^{-1}(\partial I)=\partial M$, where $I$ is the segment $f(M)$.
An arbitrarily chosen Riemannian metric on $M$  defines a gradient flow $\delta_t$ on $M$ such that the trajectory $t\mapsto \delta_t(y)$ of any point $y$ is a curve in $M$ parametrized by an open, half-open, or closed interval. In particular, the trajectory of a critical point $x\in M$ is a constant curve $t\mapsto x$ parametrized by $t\in (-\infty, \infty)$. If a trajectory $t\mapsto \delta_t(y)$ is not closed in $M$, then the limit of $\delta_t(y)$ as $t\to \infty$ or $t\to -\infty$ is a critical point of $f$. 
  For a critical point $x$ of $f$, let $D(x)$ denote the closed subset of points $y$ in $M$ such that the limit of the trajectory $t\mapsto \delta_t(y)$ as $t\to \infty$ or $t\to -\infty$ is the critical point $x$.

Given a set $I\subset \R$, we write $M_I$ for the subset in $M$ of points $x$ such that $f(x)\in I$. To simplify notation, we write $M_a$ for $M_{\{a\}}$. Similarly, the intersection of $D(x)$ with $M_a$ is denoted by $D_a(x)$. 

{\color{red} 
\iffalse 
For a pair $(a,b)$ of real numbers such that  $a <b$, the gradient flow $\delta_t$ defines a diffeomorphism 
\[
h_{a,b}\co M_a\setminus D_a\to M_b\setminus D_b
\]
 by associating to a point $x$ a unique point in $M_b\setminus D_b$ on the trajectory $t\mapsto \delta_t(x)$ of $x$. We will also write $h_{b, a}$ for the inverse of $h_{a,b}$. 
 \fi}
 
 Let $(a, b)$ be a pair of regular values of $f$ such that $a<b$. Then the gradient flow $\delta_t$ defines a diffeomorphism 
\[
h_{a,b}\co M_a\setminus (\cup D_a(x))\to M_b\setminus (\cup D_b(x))
\]
 by associating to a point $y$ a unique point in $M_b\setminus (\cup D_b(x))$ on the trajectory $t\mapsto \delta_t(y)$ of $y$. Here the unions $\cup D_a(x)$ and $\cup D_b(x)$ are taken over the critical points $x$ satisfying $a<f(x)<b$. 
  We will also write $h_{b, a}$ for the inverse of $h_{a,b}$.

We say that a critical point of $f$ is \emph{isolated} if it has a neighborhood that contains no other critical points of $f$. Such a neighborhood is called \emph{isolating}.
Following Seifert and Threlfall \cite[p. 38]{ST38}, we say that 
an isolating compact neighborhood $H$ of a critical point $x_0$ with $f(x_0)=a$ is \emph{cylindrical} if 
$H$ is disjoint from $\partial M$, 
and $H$ consists of trajectories of the flow $\delta_t$ in $M_{[a-\varepsilon, a+\varepsilon]}$. We note that each trajectory in $H$ either joins a point in $M_{a-\varepsilon}$ with a point in $M_{a+\varepsilon}$ or belongs to $D(x_0)\cap M_{[a-\varepsilon, a+\varepsilon]}$. Indeed, if $M$ is open,
this follows from the compactness of $H$, while if $M$ is compact, this follows from $H\cap \partial M=\emptyset$. 
We will always assume that a cylindrical neighborhood $H$ is a smooth manifold of dimension $\dim M$ with corners, and each of $H\cap M_{a-\varepsilon}$ and $H\cap M_{a+\varepsilon}$ is a  smooth manifold with (smooth) boundary.

\begin{remark}
In the theory of maps with isolated cone-like critical points to manifolds of arbitrary dimension adapted neighborhoods \cite{Fu21} play an important role. In general the definition of adapted neighborhoods involves a list of assumptions, but in the case of a map to $\R$ it is equivalent to the following one. An \emph{adapted} neighborhood of a cone-like isolated critical point $x$ of a continuous function $f$ is an isolating connected compact manifold neighborhood $U$ of $x$ such that the level set $f^{-1}(a)$ is transverse to $\partial U$ for each $a$ in the interior of $f(U)$. Clearly, every cylindrical neighborhood is an adapted neighborhood. 
\end{remark}

 We note that the boundary $\partial H$ of a cylindrical neighborhood $H$ is the union of the subsets $\partial_hH$ and $\partial_vH$, where the \emph{horizontal part} $\partial_hH$ of the boundary  is a disjoint union of smooth compact manifolds $H\cap M_{a-\varepsilon}$ and $H\cap M_{a+\varepsilon}$, while the \emph{vertical part}  $\partial_vH$ of the boundary  $\partial H$ is a manifold diffeomorphic to 
 \[
  \partial(\partial_hH\cap f^{-1}(a+\varepsilon))\times [a-\varepsilon, a+\varepsilon].
 \]
The subset $\partial (\partial_hH)=\partial(\partial_vH)$ of $\partial H$ is the set of corners.

Suppose now that $x_0$ is a cone-like singular point of a smooth function $f$. 
 Since $x_0$ is isolated, we may assume that $f(x_0)=0$, and $x_0$ is the only critical point on the singular hypersurface $M_0$. Similarly, we may assume that there are regular values $c$ and $-c$ of $f$ such that $0$ is a unique critical value in an open interval containing $[-c, c]$. 

 We choose a compact cone neighborhood $V$ of $x_0$ in $M_0$. 
Let $H$ denote the subset of $M_{[-c, c]}$ of points $x$ such that either the gradient curve through $x$ intersects $V$ or one of the limits of the gradient curve through $x$ is $x_0$.  By choosing $V$ and $c$ sufficiently small, we may assume that $H$ is an isolating  neighborhood of $x_0$. It immediately follows that the set $H$ is closed in $M$. We write $H_a$ for the intersection of $H$ and $M_a$.

\begin{lemma}
The subsets $H_c\subset M_c$ and $H_{-c}\subset M_{-c}$ are smooth submanifolds with (smooth) boundary. 
The subset $H\subset M$ is a smooth submanifold of codimension $0$ with corners $\partial H_{-c}\sqcup \partial H_{c}$. Furthermore, the subset $H$ of $M$ is a cylindrical neighborhood of $x_0$. 
\end{lemma}

\begin{proof} Let $x$ be a point in $H\setminus D(x_0)$. Suppose that $f(x)\in (-c, c)$. We will omit the cases $f(x)=c$ and $f(x)=-c$ as the argument in these cases  is similar. 
Without loss of generality we may assume that $f(x)\le 0$ as the case $f(x)\ge 0$ is similar.  Let $\delta$ denote the gradient flow of $f$. Then for some $t\ge 0$, the point $\delta_{t}(x)$ is  in $V$. 

Suppose that $\delta_t(x)\in \partial V$. Then for a small neighborhood $U$ of $x$ in $M$, there is a smooth map $\varphi\co U\to M_0$  that takes $y\in U$ to a unique point $\varphi(y)\in M_0$ on the trajectory of $y$. In view of the diffeomorphism $h_{f(x), 0}$,  the map $\varphi$ is a submersion onto a small neighborhood of $\delta_t(x)$. Since $V\setminus \{x_0\}$ is a manifold with smooth boundary, we conclude that
 $x$ admits a half disc neighborhood in $H$. 

In the case where $\delta_t(x)$ is a point in the interior of $V\setminus \{x_0\}$, a similar argument shows that $x_0$ admits an open disc neighborhood in $H$.

Suppose now that $x\in D(x_0)\cap H$. We still assume that $f(x)\le 0$. We need to show that for every point $y$ sufficiently close to $x$, either $y\in D(x_0)$ or the trajectory of the gradient flow of $f$ through $y$  intersects $V$, and consequently,  the point $x$ possesses an open disc neighborhood in $H$. 
Assume to the contrary that 
there is a sequence of points $y_i$ approaching $x$ such that each trajectory $t\to \delta_t(y_i)$ passes through a point $z_i$ in $M_0\setminus V$ at some moment $t_i$. We may assume $f(y_i)=f(x)$. 
Since $M_0$ is compact, there is an accumulating point $z$ of the points $z_i$ in $M_0$. 
In fact, by dropping some of the points $y_i$ (and $z_i$), we may assume that $z$ is a limit point of $z_i$.  Since $V$ is a compact neighborhood of the critical point $x_0$, the limit $z$ of points $z_i\in M_0\setminus V$ is distinct from $x_0$. 

We claim that the trajectory of the gradient flow of $f$ through the point $z$ passes through the point $x$, which contradicts the assumption that the trajectory of the gradient flow through $x$ has a limit point at $x_0$.  To prove the claim, let $y$ denote the point on the trajectory $t\mapsto \delta_{-t}(z)$ such that $f(y)=f(x)$. Then the trajectories through the points $z_i$ sufficiently close to $z$ intersect the level set $f^{-1}(f(y))$ at points close to $y$, i.e., the points $y_i$ approach the point $y$. Thus, $y$ coincides with $x$.

Finally, the maps $h_{-c,0}$ and $h_{0,c}$ define diffeomorphisms $\partial H_{-c}\to \partial V$ and $V\to \partial H_{c}$, and therefore, the subsets $H_c$ and $H_{-c}$ of smooth manifolds $M_c$ and $M_{-c}$ are smooth submanifolds with smooth boundary. \end{proof}

\begin{proposition}\label{prop:6} The manifold $H$ is a cone over $\partial H$ with vertex at $x_0$. 
\end{proposition}
\begin{proof}   Since $V$ is a cone over its boundary with vertex at $x_0$, it is a union of $x_0$ and a smooth manifold diffeomorphic to $\partial V\times [0,1]$. Let $v'$ denote a vector field on $\partial V\times [0,1]$ in the direction along the second component. Let $\lambda$ denote a smooth monotonic function on $[0,1]$ that is $0$ only at $0$ and $1$ at $1$. Let $\pi_{[0,1]}$ denote the projection $\partial V\times [0,1]\to [0,1]$. Then $(\lambda\circ \pi_{[0,1]})\cdot v'$ extends to a vector field $v$ on $V$. 

The vector field $v$ restricts to a tangent radial vector field over the manifold $V\setminus \{x_0\}$, it is trivial only at $0$,  and  any point in $V\setminus \{x_0\}$ escapes $V$ along $v$ in finite time. 

We will now use the gradient vector field $\nabla f$  to extend the vector field $v$ over $H$, see Lemma~\ref{l:11}.  For any regular point $x\in H$ of $f$, the tangent space $T_xM$ is the direct sum of the \emph{horizontal} component $T_xM_{f(x)}$ and the \emph{vertical} component generated by $\nabla f$. We say that a vector field $u$ over $H\setminus \{x_0\}$ is \emph{horizontal} if the vector $u(x)$ is in $T_xM_{f(x)}$ for all $x\in H\setminus \{x_0\}$.

\begin{lemma}\label{l:11} The vector field $v$ extends to a continuous vector field over $H$ such that $v|H\setminus \{x_0\}$ is a nowhere zero horizontal vector field.
\end{lemma}
\begin{proof} For any point $x\in H\setminus D(x_0)$ with $f(x)<0$,  
there is a unique value $t>0$ such that $\delta_t(x)$ is a point $y$ in $V$, where $\delta_t$ is the gradient flow of $f$. We define $v(x)$ by 
\[
v(x)=\pi\circ (d\delta_t)^{-1}(v(y)),
\]
 where the map $\pi$ projects the tangent space $T_yM$ at $y\in M_{[-c, c]}\setminus x_0$ to the tangent space $T_yM_{f(y)}$ of the level manifold. 
 
 We claim that $v(x)$ is non-zero. Since $\delta_t$ takes gradient curves of $f$ to gradient curves of $f$,  it takes $\nabla f(y)$  to   $\nabla f(x)$, and therefore it takes $v(y)$ to a vector linearly independent from $\nabla f(x)$, i.e., to a vector that is not in the kernel of  $\pi$. 
 
 Similarly, for $x$ in $H\setminus D(x_0)$ with $f(x)>0$ we define 
 \[
 v(x)=(\pi\circ d\delta_t)(v(y)), 
 \]
 where $x=\delta_t(y)$ for some $t>0$ and $y\in V$. Finally, we set $v(x)=0$ for 
all $x\in D(x_0)$. Then $v$ is a desired continuous vector field over $H$. 
\end{proof}

 Let $w$ denote $f\cdot \nabla f$.  Then $w(x)$ vanishes precisely over the set $M_0$ as well as over the critical points of $f$.

\begin{lemma} The vector field $v+w$ is a continuous vector field on $H$ which is nowhere zero on $H\setminus \{x_0\}$. 
\end{lemma}
\begin{proof}
Since $w$ and $v$ are continuous vector fields, we deduce that $w+v$ is a continuous vector field over $H$. It is easily verified that $w+v$ is trivial only at $x_0$. 
\end{proof}
\iffalse
\sout{We claim that $w+v$ is trivial only at $x_0$. Indeed, assume to the contrary that $x\in H$ is a regular point of $f$ with $w(x)=-v(x)$. Over $V\setminus \{x_0\}$ we have $v(x)\ne 0$ and $w(x)=0$. Over $D(x_0)\cap (H\setminus x_0)$ we have $v(x)=0$ and $w(x)\ne 0$. Suppose that $x$ is not on $D(x_0)\cup V$. We may assume that $f(x)<0$ as the case $f(x)>0$ is similar. Then for some $t>0$ the point $y=\delta_t(x)$ is in $V\setminus \{x_0\}$. The diffeomorphism $\delta_t$ takes the curve $t\mapsto \delta_t(x)$ to itself, it takes the point $x$ to $y$,  and it takes the vector $v(y)$ normal to the curve to a vector not tangent to the curve $t\mapsto \delta_t(x)$. On the other hand, the vector 
$w(x)$ is tangent to the curve $t\mapsto \delta_t(x)$. Thus $v(x)+w(x)\ne 0$.}}
\end{proof}
\fi

\begin{lemma} The corners of $H$ can be smoothened so that $v+w$ is still defined on the modified manifold $H$ and $v+w$ is nowhere tangent to $\partial H$. 
\end{lemma}
\begin{proof}
The vector field $v+w$ is continuous over $H$, and it is outward normal over $\partial H\cap M_{\pm c}$ and $\partial H\setminus M_{\pm c}$.  Near a corner point, $H$ is a quarter subspace of $\R^m$ bounded by the horizontal space $f^{-1}(c_+)$ or $f^{-1}(c_-)$,  and a vertical space composed of gradient flow curves. It follows that the corner of $H$ can be smoothened, and the vector field $v+w$ can be slightly modified near the smoothened corner so that it is outward normal everywhere over $\partial H$, see \cite{Pu02}.  
\end{proof}

To summarize,  we constructed a smooth isolating compact manifold neighborhood $H$ of $x_0$ together with a vector field $v+w$ on $H$ that is nowhere zero on $H\setminus \{x_0\}$ and outward normal over $\partial H$.

\iffalse
\begin{lemma}\label{le:10} For every point $x\in H\setminus \{x_0\}$, its flow curve along $-v-w$ approaches $x_0$. 
\end{lemma}
\begin{proof}
 Along the vector field $-w$ every point of $H$ flows towards $V$, while along the vector field $-v$ the points of $H$ stay on the same level of $f$. Consequently, any point in $H\setminus \{x_0\}$ appears arbitrarily close to $V$ after flowing along $-v-w$ for sufficiently long time. On the other hand, by the definition of $v$ any point on $V$ flows towards $x_0$ along $-v$.  
 \end{proof}
\fi

\begin{lemma}\label{le:cylinder} The manifold $H\setminus \{x_0\}$ is diffeomorphic to $\partial H\times [0, \infty)$. 
\end{lemma}

\begin{proof}   Let $\delta^{-v-w}_t$ denote the flow along the vector field $-v-w$.  Suppose that the trajectory $t\mapsto \delta^{-v-w}_t(z)$ is well-defined for $t\in [0, a]$. Since $-v-w$ is inward normal over $\partial H$, the trajectory of any point $z$ does not escape $H$, and therefore $\delta^{-v-w}_a(z)$ is in $H$. Consequently, the curve $t\mapsto \delta^{-v-w}_t(z)$ is well-defined for $t\in [0, a+\varepsilon)$ for some $\varepsilon$. Suppose now that $t\mapsto \delta^{-v-w}_t(z)$ is well-defined for $t\in [0, a)$ and it is not well-defined for $t\in [0,a]$. Since $-v-w$ is nowhere zero, we deduce that the limit of $\delta^{-v-w}_t(z)$ as $t\to a$ is $x_0$.    

Next, we claim that for every point $x\in H\setminus \{x_0\}$, its flow curve along $-v-w$ approaches $x_0$.  Indeed, along the vector field $-w$ every point of $H$ flows towards $V$, while along the vector field $-v$ the points of $H$ stay on the same level of $f$. Consequently, any point in $H\setminus \{x_0\}$ appears arbitrarily close to $V$ after flowing along $-v-w$ for sufficiently long time. On the other hand, by the definition of $v$ any point on $V$ flows towards $x_0$ along $-v$.

Let $\mu>0$ denote a function on $H\setminus \{x_0\}$ such that  $\mu(x)$ and all its derivatives tend to $0$ as $x\to x_0$. We may choose the function $\mu$ so that $\mu\equiv 1$ near $\partial H$, and every trajectory $\delta^{-\mu (v+w)}_t(z)$ of $-\mu (v+w)$ starting from any point $z\in \partial H$ is well-defined for $t\in [0,\infty)$. Then the vector field $-\mu (v+w)$ defines a diffeomorphism $\varphi\co H\setminus \{x_0\}\to \partial H\times [0, \infty)$ by $\varphi(x)=(y, t)$ for a unique point $y\in \partial H$ and a unique value $t\ge 0$ such that  $\delta_t^{-\mu (v+w)}(y)=x$. 
\end{proof}

Thus, for every point $x\in H\setminus \{x_0\}$, the flow curve through $x$ passes through a unique point of $\partial H$, and has $x_0$ as its limit point. This implies that $H$ is a cone over $\partial H$.   
\end{proof}

\begin{corollary} The manifold $H$ is a smooth manifold homeomorphic to a ball. 
\end{corollary}
\begin{proof} Since $H$ is a cone over a manifold, and $H$ is itself a manifold, it follows that $H$ is a disc, e.g. see the proof of Lemma~\ref{Kyle}. 
\end{proof}

Since the vector field $-v-w$ is tangent to $V\setminus \{x_0\}$ over $V\setminus \{x_0\}$, it follows that $(H, V)$ is a cone neighborhood pair for the critical points $x_0$. This completes the proof of Theorem~\ref{th:1}.

\section{Proof of Corollary~\ref{cor_R}}\label{s:3C4}

Let $x_0$ be an isolated critical point of a function $f\co M\to \R$ on a smooth manifold of dimension $n$. By restricting the function $f$ to a coordinate neighborhood about the critical point $x_0$, we may assume that $M$ is $\R^n$, and $f$ has no other critical points.
 In \cite{Ro50} E. Rothe studied the class of critical points satisfying the hypothesis $\mathcal{H}$.

\vskip 3mm
\noindent{\bf Hypothesis $\mathcal{H}$.} There exists an isolating neighborhood $U$ of $x_0$ such that for all $x_0\ne 0$ in $U\cap f^{-1}(x_0)$ the vectors $x-x_0$ and $\nabla f(x)$ are linearly independent.  
\vskip 3mm

The hypothesis $\mathcal{H}$ is closely related to the hypothesis $\mathcal{H}'$. 

\vskip 3mm
\noindent {\bf Hypothesis $\mathcal{H}'$.} There exists an isolating compact neighborhood $U$ of $x_0$ such that 
\begin{itemize}
\item $U$ is a submanifold of $M$ with (smooth) boundary $\partial U$, 
\item $\partial U$ is nowhere tangent to $f^{-1}(x_0)$, 
\item over the complement to $x_0$ in $V=U\cap f^{-1}(x_0)$ there is a smooth vector field $v$ which is nowhere zero and which is outward normal over $\partial V$. 
\end{itemize}

\begin{lemma} The hypothesis $\mathcal{H}$ implies the hypothesis $\mathcal{H}'$.
\end{lemma}
\begin{proof} Let $U$ be an isolating compact neighborhood of a critical point $x_0$ of a function $f$. Suppose that $U$ satisfies the hypothesis $\mathcal{H}$. 
Let $r\co U\to \R$ denote the function that associates with $x$ the distance between $x$ and $x_0$ with respect to the Euclidean metric on $\R^n$. 
We may choose a ball neighborhood $U_\varepsilon$ of $x_0$ of radius $\varepsilon$ so that $U_\varepsilon\subset U$ and $\varepsilon$ is a regular value of $r$. Put $V_\varepsilon=U_\varepsilon\cap f^{-1}(x_0)$. 

Since the vectors $x-x_0$ and $\nabla f(x)$ are linearly independent, the projection $v(x)$ of the vector $x-x_0$ to $T_xV$ is non-zero for all $x\in V\setminus\{x_0\}$.  The so-constructed vector field $v$  over $V\setminus \{x_0\}$ satisfies the hypothesis $\mathcal{H}'$.
\end{proof}

\iffalse
\begin{lemma}\label{l:6} There exists a smooth vector field $v_3$ on $V\setminus \{x_0\}$ which is nowhere zero. Furthermore, the vector field $v_3$ can be chosen so that $v_3|\partial V$ is outward normal.   
\end{lemma}
\begin{proof}  Let $r_3$ denote the smooth function on $V\setminus \{x_0\}$ given by the composition $g_3h_{0,c}$. Let $v_3$ denote the gradient vector field of $r_3$. Since $g_3$ is non-singular on $U\setminus D_c(x_0)$, the vector field $v_3$ is nowhere zero. 
 Since the gradient vector field of $g_3$ over $\partial U$ is outward normal, we deduce that $v$ is outward normal over $\partial V$. 
\end{proof}
\fi 

The proof of the following Lemma~\ref{le:13} is similar to one of Lemma~\ref{le:cylinder}.

\begin{lemma}\label{le:13} The manifold $V\setminus \{x_0\}$ is diffeomorphic to $\partial V\times [0, \infty)$. 
\end{lemma}

%We also note that the limit of $\delta^{-\lambda v}_t(z)$ as $t\to \infty$ is $x_0$ for all $z\in \partial V$. Therefore we deduce Corollary~\ref{c:8}. 

\begin{corollary}\label{c:8} The set $V$ is homeomorphic to the cone over its boundary. 
\end{corollary}

Now the proof of Proposition~\ref{prop:6} shows that a critical point satisfying the hypothesis $\mathcal{H}$ admits a  cylindrical ball neighborhood.

The above argument proves a characterization of a cone neighborhood pair for an isolated critical point. 
 
\begin{proposition}\label{prop:cone16} Let $B_\varepsilon$ be an $\varepsilon$-disc neighborhood of $x_0$ such that $\partial B_\varepsilon$ is transverse to $M_0$. Suppose that $V=B_\varepsilon\cap M_0$ is a cone over its boundary with vertex at $x_0$. Then there exists a cone neighborhood pair for $x_0$. 
\end{proposition}
\begin{proof} Since $V$ is a cone over its boundary with vertex at $x_0$, the complement $V\setminus \{x_0\}$ is diffeomorphic to $V\times (-\infty, 0]$. Consequently, there is a nowhere zero vector field $v$ on $V\setminus \{x_0\}$ that is outward normal over $\partial V$. By an argument as in the proof of Proposition~\ref{prop:6}, we conclude that the critical point $x_0$ admits a cylindrical ball neighborhood $H$. In fact, the construction in the proof of Proposition~\ref{prop:6} also produces a nowhere zero vector field $w+v$ over $H\setminus\{0\}$ that restricts to $v$ over $V$. Consequently, the pair $(H, H\cap V)$ is a cone neighborhood pair for $x_0$. 
\end{proof}

\section{Algebraic singular points}\label{s:4alg}

In this section we make precise the statement that the set of Taylor series of map germs with non cone-like  critical points is of infinite codimension in the space of all Taylor series. In subsection~\ref{s:tour} we  recall the Tourgeron theorem on finitely determined map germs and deduce that map germs with non cone-like critical points are extremely rare, see Corollary~\ref{c:t17}. In subsection~\ref{s:inf} we will discuss infinitely determined map germs, which is a larger class than the class of finitely determined map germs, and show that if a singular map germ is infinitely determined, then its critical point is cone-like, see Proposition~\ref{p:inf20}.

\subsection{Finitely determined map germs}\label{s:tour}
Let $\mathcal{E}(n, p)$ denote the vector space of map germs at $0$ of smooth functions  $\R^n\to \R^p$. The vector space $\mathcal{E}(n)=\mathcal{E}(n, 1)$ of function germs is a ring with multiplication and addition induced by multiplication and addition of functions. 
Let $\mathcal{B}(n)$ denote the subset of invertible map germs $f$ in $\mathcal{E}(n,n)$ with an additional condition $f(0)=0$. The set $\mathcal{B}(n)$ is a group with operation given by the composition. 
We say that map germs $f_0, f_1\in \mathcal{E}(n, p)$ are \emph{right equivalent} if there is a map germ $h\in \mathcal{B}(n)$ such that $f_0\circ h=f_1$. For a non-negative integer $k$, the \emph{$k$-jet} $\hat{f}$ of a map germ $f$ at $0$ is the Taylor polynomial of $f$ at $0$ of order $k$. Similarly, for $k=\infty$, the $k$-jet of $f$ is the Taylor series of $f$ at $0$. The vector space of all $k$-jets at $0$ is denoted by $\hat{\mathcal{E}}_k(n, p)$. The vector space of Taylor series of map germs $\R^n\to \R^p$ at $0$ is called the (infinite) \emph{jet space}. It is denoted by $\hat{\mathcal{E}}(n, p)\co =\hat{\mathcal{E}}_\infty(n, p)$. There is a sequence of projections of Euclidean spaces
\[
   \hat{\mathcal{E}}(n, p)\to \cdots \to \hat{\mathcal{E}}_{k+1}(n, p)\to \hat{\mathcal{E}}_k(n, p)\to \cdots 
\]
that truncate the series/polynomials.  The truncation $\hat{\mathcal{E}}(n, p)\to \hat{\mathcal{E}}_k(n, p)$ is denoted by $\pi_k$. We say that a subset $X$ in $\hat{\mathcal{E}}(n, p)$ is \emph{proalgebraic} if it is of the form $\cap \pi^{-1}_k X_k$ for some algebraic sets $X_k\subset \hat{\mathcal{E}}_k(n, p)$. By definition, the \emph{codimension} of $X$ is the upper limit of codimensions of the algebraic subsets $X_k$. 

A map germ $f$ is said to be \emph{$k$-definite} if every map germ $g$ with the same $k$-jet $\hat{g}=\hat{f}$ is right equivalent to $f$. For example, a function germ $f\co \R^n\to \R$ is $1$-determined at $0$ if and only if it is non-singular at $0$, i.e., $df(0)\ne 0$. A function germ $f\co \R^n\to \R$ is $2$-determined at a critical point $0$ if and only if $0$ is a Morse critical point. 
We say that a map germ $f\co (\R^n, 0)\to (\R^p, 0)$ is \emph{finite} if $\mathcal{E}(n)/\langle f_1,..., f_p\rangle$ is finite dimensional, where $\langle f_1,..., f_p\rangle$ is the ideal generated by the components $f_1,..., f_p$ of $f$. We note that the condition 
$\mathcal{E}(n)/\langle f_1,..., f_p\rangle=k$ implies \[
\dim \mathcal{E}(n)/(\langle f_1,..., f_p\rangle + m(n)^{k+1})=\dim \hat{\mathcal{E}}_k(n)/\langle \hat{f}_1,..., \hat{f}_p\rangle \le k,
\]
where $m(n)$ is the maximal ideal in $\mathcal{E}(n)$. The converse is also true, i.e., 
if we have $\dim \hat{\mathcal{E}}_k(n)/\langle \hat{f}_1,..., \hat{f}_p\rangle \le k$ for some $k$, then $f$ is finite. 
It is known that if $df\co \R^n\to \R$ is a finite map germ at $0$, then the map germ $f\co \R^n\to \R$ at $0$ is finitely determined, e.g., see \cite[\S11.10]{BL}.

On the other hand, let $Y$ denote the set of Taylor series of non-finite map germs. Then 
$
    Y=\cap \pi_k^{-1}Y_k
$
for the sets
\[
Y_k=\{\hat{g}\in \hat{\mathcal{E}}_k(n, p)\ |\  \dim \hat{\mathcal{E}}_k(n)/\langle \hat{g}_1,..., \hat{g}_p\rangle >k\},
\]
where $\hat{g}_1,..., \hat{g}_p$ are components of $\hat{g}$. 

\begin{theorem}[Tourgeron Theorem] The sets $Y_k$ are algebraic. In fact, if $n\le p$, then $Y$ is a proalgebraic set of infinite codimension. 
\end{theorem}

\begin{corollary} The set of jets of map germs $f\in \mathcal{E}(n)$ for which $df$ is not finite is a subset of infinite codimension. Thus, the set of jets of non-finitely determined map germs is a subset of a proalgebraic set of infinite codimension. 
\end{corollary}

For proofs of the Tourgeron Theorem and its corollary we refer to \cite[Theorem 13.4 and Remark 13.6]{BL}  respectively. 

We note that a $k$-determined map germ $f$ is right equivalent to a polynomial $g$ of degree $k$. Since the critical point of a polynomial $g$ is algebraic, we deduce that the critical point of any finitely determined map germ $f$ is cone-like. 

\begin{corollary} \label{c:t17} The set of jets of map germs with non cone-like critical points is a subset of a proalgebraic set of infinite codimension. 
\end{corollary}

\subsection{Infinitely determined map germs}\label{s:inf}
Next, we turn to infinitely determined map germs $f\co \R^n\to \R^p$ at $0$. 
We say that a map germ $f$ at $0$ is \emph{infinitely determined} if any map germ $g$ at $0$ that has the same Taylor series as $f$ is right equivalent to $f$.

Let $J_f$ denote the ideal in $\mathcal{E}(n)$ generated by $p\times p$ minors of the differential $d_0f$ of $f$ at $0$, where $f$ is a map germ at $0$ of a function $\R^n\to \R^p$. We say that a finitely generated ideal $I$ in $\mathcal{E}(n)$ is \emph{elliptic} if it contains $m_n^\infty$, where $m_n$ is the maximal ideal in $\mathcal{E}(n)$. Part of Theorem~\ref{th:8} was announced in \cite{Ku77}. For its proof we refer the reader to \cite{Br81}, \cite{N77}, and \cite{Wi81}, see also \cite[Theorem 6.1 and Lemma 6.2]{Wa81}.

\begin{theorem}\label{th:8} The following conditions are equivalent:
\begin{itemize}
\item A map germ $f$ is infinitely determined.
\item $J_f$ is elliptic.
\item $||d_xf||^2\ge C||x||^\alpha$ holds on some neighborhood of $0$ for some constants $C$ and $\alpha>0$, where  $||d_xf||^2$ is the sum of squares of $p\times p$ minors of $d_xf$. 
 \end{itemize} 
\end{theorem}

By Theorem~\ref{th:8}, if $f\co \R^n\to \R$ is not an infinitely determined function germ at $0$, then there is a sequence of points $z_1, z_2, ....$ converging to $0$ such that 	
\[
\left(\frac{\partial f}{\partial x_1}(z_i)\right)^2 + \cdots +\left(\frac{\partial f}{\partial x_n}(z_i)\right)^2=o||z_i||^N
\]
 for all $N$.

\begin{example} The map germ $f(x, y)=(x^2+y^2)^2$ is not finitely determined \cite[p. 237]{Wi82}, but it is infinitely determined. The map germ $g(x, y)=(x, y^4+xy^2)$ is not infinitely determined.  
\end{example}

\begin{proposition}\label{p:inf20} If a singular map germ is infinitely determined, then its critical point is cone-like. 
\end{proposition}
If $E$ is an equivalence relation on map germs, then we say that a map germ $f$ is $k$-$E$-determined if every map germ $g$ with the same $k$-jet $\hat{g}$ as the $k$-jet $\hat{f}$ of $f$ is $E$-equivalent to $f$. An important example of an $E$-equivalence is a $\mathcal{K}$-equivalence. If $f$ and $g$ are $\mathcal{K}$-equivalent, then there is a diffeomorphism of source spaces of $f$ and $g$ that takes $f^{-1}(0)$ to $g^{-1}(0)$.

\begin{proof} 
Suppose that $f$ is infinitely determined. Then it is infinitely $\mathcal{K}$-determined. Then, by the Brodersen theorem \cite{Br81}, it is finitely $\mathcal{K}$-determined. In particular, the map germ $f$ is $\mathcal{K}$-equivalent to a polynomial. Thus, there is a diffeomorphism of the source spaces of $f$ and $g$ that takes $f^{-1}(0)$ to an algebraic set $g^{-1}(0)$. 
\end{proof}

\iffalse
Let $\mathcal{G}$-equivalence stands for $\mathcal{R}$, $\mathcal{C}$, or $\mathcal{K}$-equivalence. 
Let $J_f$ denote the ideal in $\mathcal{E}_n$ generated by $p\times p$ minors of $df$. We say that a finitely generated ideal $I$ in $\mathcal{E}_n$ is \emph{elliptic} if it contains $m_n^\infty$, where $m_n$ is the maximal ideal in $\mathcal{E}_n$. 
Put
\[
     I_\mathcal{R}=J_f, \quad I_\mathcal{C}(f)=f^*m_p\cdot  \mathcal{E}_n, \quad I_\mathcal{K}(f)=I_\mathcal{R}(f)+I_\mathcal{C}(f). 
\]

Let $N_\mathcal{R}(f)(x)$ denote the sum of squares of $p\times p$ minors of $df(x)$. Define 
\[ 
   N_\mathcal{C}(f)(x)=||f(x)||^2, 
\]
and put $N_\mathcal{K}(f)(x)=N_\mathcal{C}(f)(x)+N_\mathcal{R}(f)(x)$.

\begin{theorem} The following conditions are equivalent:
\begin{itemize}
\item A map germ $f$ is infinitely $\mathcal{G}$-determined.
\item $I_\mathcal{G}(f)$ is elliptic.
\item $N_\mathcal{G}(f)(x)\ge C||x||^\alpha$ holds on some neighborhood of $0$ for some constants $C$ and $\alpha>0$. 
 \end{itemize} 
\end{theorem}

\fi

\section{The Takens functions}\label{s:5Tak}

In \cite{Ta68}, Takens constructed functions on manifolds of dimension $\ge 3$ with non cone-like critical points. In this section we study Takens functions, and show that the critical points of these functions still admit cylindrical neighborhoods diffeomorphic to a ball. 

Given a smooth manifold $X$ with corners and a smooth manifold $Y$ with smooth boundary, a map $f\co X\to Y$ as well as its inverse is said to be a \emph{special almost diffeomorphism} if it is continuous and restricts to a diffeomorphism $X\setminus \Sigma\to Y\setminus f(\Sigma)$ where $\Sigma$ is the set of points of the boundary of $X$ at which $\partial X$ is not smooth. An \emph{almost diffeomorphism} of manifolds with corners is a composition of special almost diffeomorphisms.

The Takens construction is based on the existence for $n\ge 4$ of a set $Q\subset \R^{n-1}\subset \R^n$ such that 

\begin{itemize}
\item there exists a continuous map $\tau\co \R^n\to \R^n$ with $\tau(Q)=0$ such that the restriction $\tau|\R^n\setminus Q$ is a diffeomorphism onto $\R^n\setminus \{0\}$, and
\item there is no compact neighborhood $U$ of $Q$ in $\R^{n-1}$ such that $U\setminus Q$ is homeomorphic with $\partial U\times (0, 1]$.  
\end{itemize}

When $n=4$, we may choose $Q\subset \R^3\subset \R^4$ to be the Fox-Artin arc, which is an embedded segment in $\R^3$ with non-trivial fundamental group $\pi_1(\R^3\setminus Q)$. When $n\ge 5$, the space $Q\subset \R^{n-1}\subset \R^n$ can be constructed by means of the Newman-Mazur compact contractible manifold $M$ of dimension $n-1$ with non-trivial $\pi_1(\partial M)$. 
The Newman-Mazur compact manifold $M$ has the properties that
\begin{itemize}
\item $M\times [0,1]$ is almost diffeomorphic to a round ball,
\item the double of $M$ is diffeomorphic to $S^n$, and \item the disc $D^{n-1}$ can be decomposed into a union $A_1\cup A_2$ of two manifolds with corners each of which is almost diffeomorphic to the Newman-Mazur manifold $M$. 
\end{itemize}
Starting with a disc $B_0$ of dimension $n-1$, Takens defines a sequence of compact $(n-1)$-manifolds $B_i$, where $B_{i+1}$ is obtained from $B_i$ by removing from $B_i$ a collar  
neighborhood of $B_i$ as well as a subset diffeomorphic to $A_1$ so that $B_{i+1}$ is diffeomorphic to the boundary connected sum $B_i\#_b A_2$, which is almost diffeomorphic to $B_i\#_b M$. Then $Q=\cap B_i$. 
It follows that $Q$ is a compact subset of $\R^{n-1}$. 

To construct a smooth function $f\co \R^n\to \R$ with an isolated non cone-like critical point $0$, we first construct a smooth function $F\co \R^n\to \R$ with $F^{-1}(0)=\R^{n-1}$ and $\partial F/\partial x_n\equiv 0$ precisely over $Q$. Then $f$ is defined by $F=f\circ \tau$. 

The function $F$ in \cite{Tak} is of the form 
\[
      F(x_1,..., x_n)=\lambda(x_1,..., x_{n-1})\cdot x_n + h(x_n), 
\]
where $\lambda(x_1,..., x_{n-1})$ and $h(x_n)$ are certain smooth functions such that $\lambda> 0$ over the complement to $Q$, $h'\ge 0$, and $h(0)=0$. Namely, let $\{Q_i\}$, with $i=1,2,...$, be a basis of compact neighborhoods of $Q$ in $\R^{n-1}$ such that $Q_i$ is in the interior of $Q_{i-1}$. Let
$\{\lambda_i\}$ denote the partition of unity of $\R^{n-1}\setminus Q$ with support of $\lambda_1$ in the complement of $Q_2$, and support of all other functions $\lambda_i$ in $Q_{i-1}\setminus Q_{i+1}$. Then $\lambda=\sum a_i\lambda_i$, where 
the strictly positive constants $a_i$ are chosen so that 
\iffalse
$a_i$ is less than the minimum of $1$ and 
\[
  \frac{||x\circ{\tau}^{-1}||^{2}} {2^{i}||(\lambda_i\cdot x_n)\circ \tau^{-1}||_{C^i} }, 
\]
where $||x\circ{\tau^{-1}}||^2=x_1^2\circ{\tau}^{-1}+\cdots + x_n^2\circ{\tau}^{-1}$, and $||g||_{C^i}$ is the maximum of values of derivatives of $g$ of orders at most $i$ over the compact support of the smooth function $g$ restricted to the disc $||x||\le 1$. 
Therefore, 
\fi
the series $\sum (\lambda_i\cdot x_n)\circ \tau^{-1}$ converges to a smooth function $(\lambda\cdot x_n)\circ \tau^{-1}$ whose derivatives of any order are obtained by term by term differentiating the series. 

Let $\{h_j\}$ denote a sequence of functions on $\R^1$ parametrized by $j=1,2, ...$ such that $h_j(x)=x$ if $|x|\ge 1/i$, $h_i(x)=0$ if $|x|<1/(i+1)$ and $h'_j(x)\ge 0$.  
As above, we define a function $h(x_n)=\sum b_j h_j\circ x_n\circ {\tau^{-1}}$, where the strictly positive constants $b_j$ are chosen so that the function $h$ is smooth and its derivatives are obtained by term by term differentiating the series. 
\iffalse
 $b_j$ is less than the minimum of $1$ and 
\[
   \frac{||x||^2}{2^j||h_j\circ x_n\circ \tau^{-1}||_{C^i}}. 
\]
\fi 

We immediately deduce the following proposition. 

\begin{proposition}\label{p:23Tak}  The Takens function germs $f$ are flat, i.e., the derivatives of $f$ of any order at $0$ are trivial. 
\end{proposition}

\begin{proposition}\label{prop:21} The critical points of Takens functions admit cylindrical neighborhoods diffeomorphic to a ball. 
\end{proposition}
\begin{proof}
To begin with let us show that the critical points of Takens functions admit cylindrical ball neighborhoods. 

Since $F|Q\equiv 0$, and $\partial F/\partial x_n>0$ over the complement to $Q$, the implicit function theorem implies that for each $\varepsilon\ne 0$ there is a function $g_\varepsilon(x_1, ..., x_{n-1})$ such that $F(x_1,..., x_{n-1}, g_\varepsilon(x_1,..., x_{n-1}))=\varepsilon$. In other words, the level set $F(x_1,..., x_n)=\varepsilon$ is the graph of the function $x_n=g_\varepsilon(x_1,..., x_{n-1})$. 
Furthermore, let $B_r$ denote a closed ball in $\R^{n-1}$ centered at $0$ of radius $r$.  Then the absolute value of the restriction $g_\varepsilon|B_r$ is bounded by a constant $c_\varepsilon$ such that $c_\varepsilon\to 0$ as $\varepsilon\to 0$. 

For any subset $X$ of $\R^{n-1}$, let $C_\varepsilon(X)$ denote the compact rectangular cylindrical subset of $\R^n$ over $X$ bounded by the graphs of the functions $g_{-\varepsilon}|X$ and $g_{\varepsilon}|X$. In other words, the set $C_\varepsilon(X)$ consists of points $(x_1,..., x_{n})$ such that $(x_1,..., x_{n-1})$ is a point in $X$ and 
\[
   g_{-\varepsilon}(x_1,..., x_{n-1})\le x_n \le g_{\varepsilon}(x_1,..., x_n). 
\] 
Then $C_\varepsilon(X)$ is diffeomorphic to the cylinder $X\times [-\varepsilon, \varepsilon]$.   

Suppose that $B_r$ contains the compact set $Q$ in its interior. Let $H(r, \varepsilon)$ denote a neighborhood of $Q$ that consists of all trajectories of the gradient vector field of $F$ in $F^{-1}[-\varepsilon, \varepsilon]$ that either contain a point in $B_r$ or has a limit point in $Q$.

\begin{lemma}\label{l:cr32} There is a number $r_3>0$ such that the cylindrical neighborhood $H(r_3, \varepsilon)$ is diffeomorphic to the rectangular cylinder $C_\varepsilon(B_{r_3})$. 
\end{lemma}
\begin{proof}
Let $r_0<r_1 < r_2<r_3<r_4$ be numbers such that $B_{r_0}$ contains $Q$ in its interior, and 
$H(r_i, \varepsilon)\subset C_\varepsilon(B_{r_{i+1}})$ for $i=0,..., 3$. 
There is a smooth homotopy $v_t$ parametrized by $t\in [0,1]$ of the gradient vector field $v_0$ of $F$ to a vector field $v_1$ such that $v_t$ is constant over $H(r_0, \varepsilon)$, $v_t\cdot \partial/\partial x_n>0$ over $C_\varepsilon(B_{r_3})\setminus H(r_0, \varepsilon)$ for all $t$, and $v_1$ is vertical up over $C_{\varepsilon}(B_{r_4}\setminus B_{r_2})$. 

Given a point $x$ in $\R^{n-1}\setminus Q$, there is a unique trajectory $\gamma_x$ of the vector filed $v_t$ passing through $x$. Given $s>0$, let $y\in \R^n$ denote the point in $\R^n$ on the trajectory $\gamma_x$ such that $F(y)=s$. Then $j_{t, s}\co \R^{n-1}\setminus Q\to \R^n$ defined by $x\mapsto y$ is a smooth map that takes $B_{r_3}\setminus Q$ to the upper horizontal part $\{F\equiv s\}\cap H(r_3, s)$ of the boundary of $H(r_3, s)$.

Then there is  a smooth isotopy $\varphi_t$ of a neighborhood $H(r_3, \varepsilon)$ of $Q$ to  $C_\varepsilon(B_{r_3})$ defined by 
\[
\varphi_t(x_1,..., x_n) = j_{t, F(x_n)} (j_{0, F(x_n)}^{-1}(x_1,..., x_n)).
\]
We claim that $\varphi_1$ is a diffeomorphism onto $C_\varepsilon(B_{r_3})$. Indeed, let $(x_1,..., x_{n})$ be a point in $C_\varepsilon(B_{r_3})$ that is not on the trajectories of $v_1$ with limit points on $Q$. 
Then the point $(x_1,..., x_{n})$  is the image of the point 
 \[
     j_{0, F(x_n)}\circ j_{1, F(x_n)}^{-1}(x_1,..., x_n).
 \]
On the other hand, if $(x_1,..., x_n)$ is on a trajectory of $v_1$ with a limit point on $Q$, then it is in $H(r_0, \varepsilon)$ where the isotopy $v_t$ is constant. Therefore, the point $(x_1,..., x_n)$ in this case is the image of itself.
\end{proof}  

Since $H=H(r_3, \varepsilon)$ is diffeomorphic to $C_\varepsilon(B_{r_3})$, we conclude that $H$ is a ball with corners $B_{r_3}\times [-\varepsilon, \varepsilon]$. Consequently, $H/Q$ is a cylindrical ball neighborhood of the critical point of $f$. 

 Now let us show that the cylindrical neighborhood $H/Q$ is not only homeomorphic to a ball, but it is almost diffeomorphic to a ball. Indeed, if $n\ne 4, 5$, then every smooth manifold homeomorphic to a ball is diffeomorphic to ball. Suppose that $n=5$.  Since the boundary of $H/Q$ is diffeomorphic to the boundary of $H$, which is almost diffeomorphic to a sphere, we conclude that $H/Q$ is almost diffeomorphic to a ball, see \cite[Proposition C]{Mi65}. Suppose now that $n=4$. Then there is an ambient  isotopy of $H$ that is trivial near the boundary, and that takes the wild arc $Q$ into the unknotted standard segment $I$. Consequently, the manifold $H/Q$ is diffeomorphic to the manifold $H/I$, which is almost diffeomorphic to a smooth $4$-ball. 
\end{proof}

\begin{proof}[Proof of Theorem~\ref{th:Tak}] Since there is no compact neighborhood $U$ of $Q$ in $\R^{n-1}$ such that $U\setminus Q$ is homeomorphic with $\partial U\times (0, 1]$,  
 the critical point of a Takens function is not cone-like. On the other hand, by Proposition~\ref{prop:21}, such a critical point admits a cylindrical neighborhood diffeomorphic to a ball. 
\end{proof} 

\begin{corollary}\label{c:Tak26}  The critical points of Takens functions are removable, i.e., for any neighborhood $U$ of a critical point $x_0$ of a Takens function $f$, there is a smooth function $g$ such that $g\equiv f$ outside of $U$ and $g$ has no critical points in $U$. 
\end{corollary} 
\begin{proof} In the proof of Proposition~\ref{prop:21} we may choose $\varepsilon, r_0, r_1, r_2, r_3,$ and $r_4$ so small that $H/Q\subset U$. We note that the boundary of the ball $H/Q$ is diffeomorphic to the boundary of the ball $H$, and the function $f$ near $\partial (H/Q)$ can be identified with the function $F$ near $\partial H$. On the other hand, the boundary $\partial H$ consists of the upper and lower horizontal parts $\partial_hH=\partial H\cap \{F=\pm \varepsilon\}$, and the vertical part $\partial_vH=\overline{\partial H\setminus \partial_hH}$ such that $F$ is constant over each component of $\partial_hH$ and  non-singular over $\partial_vH$.  By Lemma~\ref{l:cr32}, the ball $H$ is a cylinder over a standard ball in $\R^{n-1}$. Therefore, the upper and lower horizontal parts of $\partial_hH\cong \partial_h(H/Q)$  are also standard balls. Therefore, there is a function $f'$ on $H/Q$ such that $f'$ agrees with $f$ near $\partial (H/Q)=\partial H$ and $f'$ has no critical points. Finally, we define $g$ to be a function on $\R^{n}$ such that $g\equiv f'$ over $H/Q$ and $g\equiv f$ otherwise.  
\end{proof}

\section{Singhof-Takens fillings}\label{s:stf2}

A continuous map $f\co X\to Y$ of a subset $X\subset \R^n$ to a subset $Y\subset \R^n$ is said to be \emph{smooth} if it extends to a smooth map from an open neighborhood of $X$ to an open neighborhood of $Y$. A smooth map $f\co X\to Y$ of subsets of $\R^n$ is a \emph{diffeomorphism} if it is a homeomorphism and its inverse is smooth.

Let $Q$ be a topological space. A \emph{coordinate $n$-chart} $(U, \varphi)$ on $Q$ consists of an open subset $U$ of $[0, \infty)^k\times \R^{n-k}$ for some $k\in \{0, ..., n\}$ and 
a homeomorphism $\varphi\co U \to Q$ onto the image. An  \emph{$n$-atlas} on $Q$ is a maximal collection $\{(U_\alpha, \varphi_\alpha)\}$ of $n$-charts  on $Q$ such that  $Q=\cup \varphi_\alpha(U_\alpha)$ and 
the transition maps
\[
     \varphi_{\beta}^{-1}\circ \varphi_{\alpha}\co \varphi_{\alpha}^{-1}(\varphi_\alpha(U_\alpha)\cap \varphi_{\beta}(U_\beta))\longrightarrow  \varphi_{\beta}^{-1}(\varphi_\alpha(U_\alpha)\cap \varphi_{\beta}(U_\beta))
\]
are diffeomorphisms for all $\alpha$ and $\beta$. A  second-countable Hausdorff space $Q$ with a maximal $n$-atlas is said to be a \emph{manifold with corners} of dimension $n$ if $k>1$. If a point $x$ in a manifold with corners belongs to the image $\varphi(U)$ of a chart of the form $U\subset [0,\infty)^k\times \R^{n-k}$, then we say that the order of the corner at $x$ is at most $k$. It follows that the set of points of order $k$ in a manifold with corners $Q$ is a smooth manifold itself. We call this smooth manifold a \emph{face} of dimension $n-k$. In particular, the face of dimension $n$ is the interior of $Q$, while all other faces are parts of the boundary $\partial Q$.  Faces of dimension less than $n$ form a \emph{canonical} stratification of the boundary $\partial Q$.

Let $Q$ be a manifold with corners. Let $\Sigma\subset \partial Q$ denote the set of points of the boundary at which the boundary is not smooth. Suppose that $Q'$ is a smooth manifold with (smooth) boundary.  Let $f\co Q\to Q'$ be a homeomorphism which restricts to a diffeomorphism $Q\setminus \Sigma\to Q'\setminus f(\Sigma)$. Then we say that $f$ as well as $f^{-1}$ is a \emph{special almost diffeomorphism}. In particular, every diffeomorphism is a special almost diffeomorphism.  %An \emph{almost diffeomoprhism} of manifolds with corners is a composition of special almost diffeomorphisms and diffeomorphisms. 

A filling of a smooth manifold $M$ is a certain decomposition of $M$ into manifolds with corners $Q_1, ..., Q_N$. The number $N$ is said to be the \emph{order} of the filling. Fillings of order $3$ of closed manifolds were introduced by Takens in \cite{Ta68}, and later the notion was extended to that of arbitrary order by Singhof~\cite{Si79}. 

\begin{definition}\label{Filling}
Let $M$ be a closed smooth manifold of dimension $n$. A \emph{Singhof-Takens filling} of $M$ is a family of compact codimension $0$ submanifolds with corners $Q_1,..., Q_N$ of $M$  such that the following conditions are satisfied.
\begin{itemize}
\item[P1:] $M=Q_1\cup \cdots \cup Q_N$, i.e, $\{Q_i\}$ is a covering of $M$ by compact subsets.
\item[P2:] The interiors of submanifolds $Q_i$ and $Q_j$ are disjoint for all $i\ne j$.
\item[P3:] Given a point $z\in M$, let $k_1<\cdots <k_\nu$ be the indices $i$ of submanifolds $Q_i$ containing $z$ such that $\nu\le n+1$. If $\nu\ge 2$, then we require that there is a coordinate
$n$-chart $D$ of dimension $n$ about $z$ in $M$ such that for each $j<\nu$
\iffalse
\[
    D\cap Q_{k_1}=\{x_1\le 0\},
\]
\[   
    D\cap Q_{k_2}=\{x_1\ge 0, x_2\le 0\},
\]
\[
 \cdots
\]
\[
   D\cap Q_{k_{\nu-1}}=\{x_1\ge 0, ..., x_{\nu-2}\ge 0, x_{\nu-1}\le 0\},  
\]
\[
   D\cap Q_{k_{\nu}}=\{x_1\ge 0, ..., x_{\nu-1}\ge 0\}. 
\]
\fi
\[
   D\cap Q_{k_{j}}=\{x_1\ge 0, ..., x_{j-1}\ge 0, x_{j}\le 0\},  
\]
and
\[
   D\cap Q_{k_{\nu}}=\{x_1\ge 0, ..., x_{\nu-1}\ge 0\}. 
\]
 \end{itemize} 
We say that $D$ is a \emph{special} coordinate chart about the point $z$ with respect to the filling $\{Q_i\}$. 
 \end{definition}
 
 We will refer to Singhof-Takens fillings as fillings for short. 
 In this paper we will consider only smooth fillings, i.e., fillings by smooth submanifolds $Q_i$ with corners.

 \begin{remark} The submanifolds $Q_i$ in a filling $\{Q_i\}$ are ordered. In fact,  the order of submanifolds $Q_i$ is essential. In particular, if $Q_1,..., Q_N$ is a filling of a manifold $M$, then $Q_{\sigma(1)}, ..., Q_{\sigma(N)}$ may not be a filling, where $\sigma$ is a permutation of $N$ elements. %Furthermore, the boundary of $Q_1$ is always smooth,  while the boundary of $Q_i$ may have corners of the form $\R^j_+\times \R^{n-j}$ for $j\le i$. 
\end{remark}

We say that a filling that consists of $N$ submanifolds $Q_i$ is \emph{categorical} if $N-1=\cat(M)$, and each submanifold $Q_i$ is contractible in $M$. 
It is known (e.g., see \cite[Proposition 1.10]{CLOT}) that if $X$ is a normal ANR, then the closed Lusternik-Schnirelmann category $\cat^{cl}X$ of $X$ agrees with the standard (open) Lusternik-Schnirel\-mann category, i.e., $\cat X =\cat^{cl}X$. Thus, if a manifold $M$ admits a filling of order $N$ by contractible submanifolds, then $N+1\ge \cat(M)$. On the other hand, every closed manifold admits a categorical filling, see \cite[Proposition 3.5]{Si79}, and \cite[Proposition 2.4]{Ta68}, and therefore the inequality is sharp for contractible fillings. 

We will also need the notion of a relative filling. 
Let $M$ be a compact manifold of dimension $n$ with a non-empty boundary $\partial M$. Then the union 
\[
\tilde M= M \sqcup (\partial M\times [0,1])/_{\partial M\sim \partial M\times \{0\}}
\]
has a unique smooth structure that agrees with the smooth structures on $M$ and $\partial M\times [0,1]$.

\begin{definition}\label{RelFilling}  Let $M$ be a compact manifold with boundary of dimension $n$. We say that a family $Q_1,..., Q_N$ of compact submanifolds with corners of $M$ is a \emph{relative filling} of $M$ if the decomposition $\tilde M=Q'_1\cup Q'_2\cup \cdots \cup Q'_N\cup Q'_{N+1}$, where $Q_i'=Q_{i-1}$ for $i=2,..., N+1$, and $Q'_{1}=\partial M \times [0,1]$, satisfies the properties $(P1)$ and $(P2)$ of Definition~\ref{Filling} as well as the property $(P3)$ for all points $z\in \tilde M\setminus \partial \tilde M$. 
We say that the relative filling $Q_1\cup\cdots \cup Q_N$ of $M$  is \emph{categorical} if $Q_i$ is contractible for $i=1,..., N$, and $N-1=\cat(M)$.
\end{definition}

\begin{remark}
We note that $Q'_{1}=\partial M \times [0,1]$ is not necessarily contractible.
\end{remark}

Many properties of relative fillings of compact manifolds with boundary are similar to the corresponding properties of fillings of closed manifolds. 
\begin{lemma} Every compact manifold $M$ admits a categorical filling. 
\end{lemma}
\begin{proof} The proof is similar to the one of  \cite[Proposition 3.5]{Si79}. 
\end{proof}

Suppose now that a manifold $M=M_1\cup M_2$ is obtained from two compact manifolds with boundary $M_1$ and $M_2$ by identifying some of their boundary components. In general, the union of a filling of $M_1$ and a filling of $M_2$ do not comprise a filling of $M$.  However, 
 the union of fillings of $M_1$ and $M_2$ do comprise a filling  of $M$ in certain cases.

 \begin{remark}\label{r:8} Let $M_1$ and $M_2$ be two compact manifolds with fillings. For $i=1,2$, let $\partial_iM$ denote a union of some of the path components of the boundary of $M_i$. Suppose that $\partial_1M$ is diffeomorphic to $\partial_2M$. Furthermore, suppose that the filling of $M_2$ is given by a decomposition $Q_1\cup \cdots \cup Q_N$ into submanifolds with corners such that if two points $x, y$ belong to the same path component of $\partial _2M$, then $x$ and $y$ belong to the same submanifold with corners $Q_i$. 
  Let $M$ be a manifold obtained from $M_1$ and $M_2$ by identifying $\partial_1M$ with $\partial_2M$. 
 Then the union of fillings of $M_1$ and $M_2$  comprise a filling of $M=M_1\cup M_2$ at least for some choice of indexing of the submanifolds with corners of $M_1\cup M_2$.  
 \end{remark}

\section{The Lusternik-Schnirelmann estimate by $\Bcat^\T$}\label{s:ls5}

Given a closed smooth manifold $M$, we recall that $\Bcat^\T(M)$ is the least number $n$ such that there is a Singhof-Takens filling of $M$ by $n+1$ topological balls. We also recall that $\Critb(M)$ is the least number such that there is a smooth function on $M$ with $\Critb(M)$ critical points each of which admits a cylindrical ball neighborhood.

In this section we will establish the Lusternick-Schnirelmann type inequality:
\[\Bcat^\T (M) + 1\le \Critb(M)\]

%better to add in thesis.
%A corner on a manifold $M$ is a point at which $\partial M$ is not smooth. In general, we can describe the neighborhood around any point as $[0,\infty)^k\times \R^{n-k}$. For example in $[0,1]\times D^2$, if our point is in the interior then $k=0$, if our point is on the interior of $\{0,1\}\times D^2$ then $k=1$, and if our point is on the boundary of $\{0,1\}\times D^2$ then $k=2$. We note that the definition of a filling requires the elements of the filling to be smooth manifolds with corners. So we need the filling of $M$ to be of smooth manifolds with corners that are homeomorphic to a ball.\\

\begin{lemma}\label{Kyle}
	Let $f:M\to \R$ be a smooth function on a closed manifold of dimension $n$. Suppose that $x_0$ is a critical point at which $f$ assumes a unique global minimum. Let $c$ be a number such that $f$ has no critical values in the interval $(f(x_0),c]$. Then the compact manifold $M_{(-\infty,a]}$ is homeomorphic to a disc for any $a\in(f(x_0),c)$.
\end{lemma}
\begin{proof}
	The manifold $M_{(-\infty,a]}$ is contractible as there is a deformation of $M_{(-\infty,a]}$ to a point along the negative to the gradient flow of $f$. In fact the manifold $M_{(-\infty,a]}$ is homeomorphic to a cone over $f^{-1}(a)=M_{a}$. Since $M$ is a manifold of dimension $n$, the group $H_i(M, M\setminus \{x_0\})=H_i(M_{(-\infty,a]}, M_{(-\infty,a]}\setminus \{x_a\})$ is isomorphic to $\Z$ in dimensions $i=0$ and $i=n$ and it is trivial otherwise. By the long exact sequence, we deduce that $M_{a}$ is a homology $(n-1)$-sphere. If $n<3$, then $M_a$ is a sphere, and therefore  $M_{(-\infty,a]}$ is homeomorphic to a disc. Suppose that $n\ge 3$.  
	
	Let us show that $M_a$ is simply connected, and, consequently, $M_a$ is a homotopy sphere. 
	To this end, let $\gamma$ be any loop on $M_a$. We will show that it is contractible in $M_a$.  Since $M_{(-\infty,a]}$ is contractible, there is a disc $D$ in $M_{(-\infty,a]}$ bounded by $\gamma$. Furthermore, since $n\ge 3$, by slightly perturbing $D$ we may assume that $x_0$ is not in $D$. Since $M_{(-\infty,a]}$ is homeomorphic to a cone over $M_a$ there is a projection of $M_{(-\infty,a]}\setminus \{x_0\}$ to $M_a$. It takes the disc $D$ to  $M_a$ producing a map of a disc extending the inclusion of $\gamma=\partial D$. Therefore, the manifold $M_a$ is simply connected. Thus $M_a$ is a homotopy $(n-1)$-sphere. Finally, by the generalized topological Poincar\'e conjecture $M_a$ is homeomorphic to a sphere, and, therefore, $M_{(-\infty,a]}$ is homeomorphic to a disc. 
\end{proof}

\iffalse
\textcolor{red}{Delete these theorems if we can smooth $H$ by taking a disc around it.(Since then we don't need to smooth $H$ like Takens does)}

\begin{theorem}
	(Douady)
	Let $V$ be a manifold with boundary and $S$ be the set of points where $\partial V$ is not smooth. If every point of $S$ has a neighborhood in $V$, diffeomorphic with $\{(x_1,\ldots,x_n)\in \R^n|x_1\geq 0 \& x_2\geq 0\}$, then $V$ has a unique smoothing $\tilde{V}$. 
\end{theorem}

\begin{theorem}\label{Noah}
	(Takens) Let $Q_1,Q_2,Q_3$ be a filling for a manifold $M$. There is a homeomorphism $h:M\to M$, isotopic with the identity such that:
	\begin{itemize}
		\item[i] $\{h(Q_i)\}_{i=1}^3$ is a filling of $M$ with $\partial Q_2$ smooth,
		\item[ii] $h(Q_i)\overset{a}{\simeq} Q_i$ for $i=1,2,3$.
	\end{itemize}
\end{theorem}

\begin{corollary}\label{Jasper}
	Let $Q_1,Q_2,\ldots, Q_n$ be a filling for a manifold $M$. There is a homeomorphism $h:M\to M$ isotopic to the identity such that:
	\begin{itemize}
		\item[i] $\{h(Q_i)\}_{i=1}^n$ is a filling of $M$ with $\partial Q_1$ smooth,
		\item[ii] $h(Q_i)\overset{a}{\simeq} Q_i$ for $i=1,2,\ldots,n$.
	\end{itemize}
\end{corollary}
\begin{proof}
	...
\end{proof}

\fi

	Let $M$ be a compact manifold with corners, and let $Q_1,\ldots, Q_{\ell}$ be a filling of $M$. We denote by $Q_{i,j}$ the union of $(j+1)$-faces of the element $Q_i$ of the filling, i.e., $Q_{i,0}$ is the complement in $\partial Q_i$ to the corners of $\partial Q_i$, and $Q_{i,j+1}$ is the complement in $\partial Q_i\bs \left(\bigcup\limits_{k\le j} Q_{i, k}\right)$  to the corners in $\partial Q_i\bs \left(\bigcup\limits_{k\le j} Q_{i, k}\right)$ for $j\ge 0$.

\begin{lemma}\label{Emma}
	Let $N$ be a smooth closed manifold of dimension $d$. Let $Q_1,\ldots, Q_{\ell}$ be a filling of $N$. Let $K$ be a smooth compact submanifold of $N$ of dimension $d$  such that $\partial K$ is transverse to each stratum $Q_{i,j}$ of the canonical stratification of each 
$\partial Q_i$. Let $Q_{r_1},..., Q_{r_n}$ be elements of the filling such that $Q_{r_i}\not\subset K$ for $i=1,..., n$, and $r_i<r_j$ for $i<j$. Put $S_i=\overline{Q_{r_i}\bs K}$ for $i\in\{1,2,\ldots, n\}$. Then the sets $K, S_1, \ldots S_n$ comprise a filling of $N$.
\end{lemma}

\begin{proof}  We note that  the sets $K, S_1, ..., S_n$ cover the manifold $N$, and the interiors of the sets $K, S_1,..., S_n$ are disjoint. Therefore, the sets $K, S_1,..., S_n$ satisfy the first two properties of a filling. 
	
Let us now verify Property 3. To begin with, let $z$ be a point in $N\setminus K$. Suppose that $z\in Q_{i_1}\cap Q_{i_2}\cap\ldots\cap Q_{i_m}$ for $1<m\leq n$, but $z$ is not in $Q_{i_1}\cap Q_{i_2}\cap\ldots\cap Q_{i_m}\cap Q_{i}$ for all $i\notin \{i_1,..., i_m\}$. Since $\{Q_i\}$ is a filling, we deduce that there is a small coordinate disc neighborhood $D$ about $z$ in $N\setminus K$ such that 
\begin{align*}
	&D\cap Q_{i_1}=\{x_1\leq 0\},\\
	&D\cap Q_{i_2}=\{x_1\geq 0, x_2\leq 0\},\\
	&\cdots\\
	&D\cap Q_{i_m}=\{x_1\geq 0, x_2\geq 0, \ldots, x_{m-1}\geq 0\}.
\end{align*}
Since $D\cap Q_i=D\cap S_i$, we deduce that Property 3 is satisfied for the point $z$. If $z$ is in the interior of $K$, then Property 3 is again clearly satisfied. 

Suppose now that $z$ is on the boundary of $K$, the point $z$ is in $Q_{i_1}, ..., Q_{i_m}$, and $z\notin Q_i$ for $i\notin\{i_1,..., i_m\}$. Again, there is a coordinate disc neighborhood $D$ about $z$ such that the intersections of $D$ with $Q_{i_1}, ..., Q_{i_m}$ are as above. The dimension of the manifold 
\[
Q = \{ x\in Q_{i_1}\cap Q_{i_2}\cap\ldots\cap Q_{i_m} \ |\ x\notin Q_{i_1}\cap Q_{i_2}\cap\ldots\cap Q_{i_m}\cap Q_{i} \quad \mathrm{for}\quad i\notin \{i_1, ..., i_m\}\}
\]
is $d-m+1$. Therefore, if $z\in \partial K$, then $d\ge m$, since $\partial K$ is transverse to $Q$.
% That is, K would have to intersect non-transversally to get $z=K\cap Q_1\cap \ldots \cap Q_{d+1}$. We don't want this because we wouldn't have enough coordinates in our coordinate chart to accommodate K since then we would have d+2 fillings which is more than what we can consider. This was an important note to me since I thought there might be an issue with the dimension and the size of the filling. If it goes above d+1 then it will break the definition since $D$ is a chart of the same dimension of the manifold. We would have $D$ intersecting $d+2$ fillings and thats bad.
Without loss of generality, we may assume that there is a smooth function $y_1$ on $D$ such that $D\cap K$ is given by $y_1\le 0$ and the gradient of $y_1$ at $z$ is non-trivial. We write $A\pitchfork_S B$ when a manifold $A$ is transverse to a manifold $B$ at each point of a set $S$.  We have \begin{align*}
	\partial K\pitchfork_z (Q_{i_1}\cap\ldots\cap Q_{i_m}) &\iff \partial K\pitchfork_z \{(x_1,x_2,\ldots, x_m)|x_1=0,\ldots,x_{m-1}=0\}\\ &\iff \nabla y_1 \notin \left\langle\pa{}{x_1},\ldots, \pa{}{x_{m-1}}\right\rangle.
\end{align*}
Indeed, since $\partial K$ is of dimension $d-1$, it is not transverse to $Q$ at $z$ only if $T_zQ\subset T_z\partial K$, i.e., $\nabla y_1$ is perpendicular to $T_zQ=\left\langle\pa{}{x_m},\ldots, \pa{}{x_{d}}\right\rangle$ with respect to the standard Riemannian metric on $D$. 

Put $y_j=x_{j-1}$ for $j=2,\ldots, m-1$. Since $\nabla y_1 \notin \left\langle\pa{}{x_1},\ldots, \pa{}{x_{m-1}}\right\rangle$, there exist functions $y_{m+1},\ldots, y_d$ of $(x_m,\ldots, x_d)$ such that the $\nabla y_1, \nabla y_m,\ldots, \nabla y_d$, are linearly independent at the point $z$. 
By the inverse function theorem, for a small disc neighborhood $D_0$ of $z$, the map $y|D_0: D_0\to \R^d$ is a local coordinate chart. %If necessary, we may choose a smaller disc $D_0\subset D$  such that the smooth map $y: D_0\to \R^d$ is a diffeomorphism onto its image, where $y=(y_1,... , y_d)$.
Then
\begin{align*}
	&D_0\cap K=\{y_1\leq 0\}\\
	&D_0\cap S_{i_1}= D_0\cap \overline{(Q_{i_1}\setminus K)}=\{y_1\geq 0, x_1\le 0 \} =\{y_1\geq 0, y_2\leq 0\}\\
	&D_0\cap S_{i_2}=D_0\cap \overline{(Q_{i_2}\setminus K)}=\{y_1\geq 0, x_1\geq 0, x_2\le 0 \}=\{y_1\geq 0, y_2\geq 0, y_3\leq 0\}\\
	&\cdots\\
	&D_0\cap S_{i_m}=D_0\cap \overline{(Q_{i_m}\setminus K)}=\{y_1\geq 0, x_1\geq 0,..., x_{m-1}\geq 0 \}=\{y_1\geq 0, y_2\geq 0, \ldots, y_m\geq 0\}
\end{align*}
which fulfills property 3.
Therefore, the sets $K, S_1,\ldots, S_n$ comprise a filling. 
\end{proof}

We are now in position to prove the following theorem. 
\begin{theorem}\label{th:13lower}
For any closed manifold $M$, we have $\Bcat^\T (M) + 1\le \Critb(M)$.
\end{theorem}

\begin{proof}
Let $x_0, x_1, \ldots, x_n$ be the critical points of a smooth function $f:M\to \R$ minimizing $\Critb(f)$.  Since $M$ is closed we may assume that $f>0$.  By slightly perturbing $f$ and re-indexing $x_0, x_1, \ldots, x_n$, we may also assume that $f(x_i)<f(x_{i+1})$. For $i=0,..., n+1$, let $a_i$ be real numbers such that $a_i<f(x_i)<a_{i+1}$ and $0=a_0$. We note $f$ is constant over $M_{a_i}$  and for each $i$ the restriction $f|{M_{[0, a_i]}}$ is a function minimizing $\Critb(f|M_{[0, a_i]})$ among functions on $M_{[0, a_i]}$ that are constant and maximal over $M_{a_i}$.

The function $f$ possesses only one critical point in the manifold $M_{[0, a_1]}$ at which $f$ assumes the global minimum. Therefore, by Lemma~\ref{Kyle}, the manifold  $M_{[0,a_1]}$ is homeomorphic to a disk.  Thus $\Bcat^\T(M_{[0,a_1]})+1=\Critb(M_{[0,a_1]})$.
By induction, suppose that 
\[
\Bcat^\T(M_{[0,a_i]})+1\leq \Critb(M_{[0,a_i]}).
\]
We note that 
\[
\Critb(M_{[0,a_i]})+1= \Critb(M_{[0,a_{i+1}]}).
\]
Let's show that $\Bcat^\T(M_{[0,a_{i+1}]})\le \ell+ 1$, where $\ell=\Bcat^\T(M_{[0,a_i]})$. Suppose that $Q_0,..., Q_\ell$ is a relative filling for $M_{[0,a_{i}]}$. It suffices to show that there is a filling for $M_{[0,a_{i+1}]}$ with $\ell+2$ elements. For an isolated critical point $x_i\in M_{{[a_i,a_{i+1}]}}$, recall that $D(x_i)$ denote the set of points $x$ in $M$ such that $\lim \gamma_t(x)=x_i$ as $t\to \infty$ or $t\to -\infty$, where $\gamma_t(x)$ is the gradient flow of $f$.  
Let $H$ be a cylindrical ball neighborhood of $x_i$ in  $M_{[a_i, a_{i+1}]}$. We may choose $H$ so that $H_{a_i}$ and $H_{a_{i+1}}$ are the upper and lower horizontal parts of the boundary of $H$. 

%Namely, we first construct a certain small neighborhood $V$ of $x_i$ in $M_{x_i}$, and then define $H$ to be the space of all points $x$ such that either the gradient curve passing through $x$ intersects $V$, or one of the limits of the gradient curve through $x$ is $x_i$. By \cite[Lemma 9]{Sa20} and \cite[Corollary 15]{Sa20} the manifold $H$ is a smooth manifold with corners, and $H$ is homeomorphic to a ball.

%justification for $\alpha$
%Suppose that $x$ is a point in $M_{[a_i, a_{i+1}]}$. By the gradient flow $\gamma$, we know there exists a unique value $t>0$ and a unique $y\in M_{a_i}$ such that $\gamma_t(y)=x$. From this we can have a diffeomorphism that associates the pair $(y,t)$ to x. Define a function $\alpha:M_{a_i}\to \R$ that takes a point $y\in M_{a_i}$ to $t\in \R$ if $\gamma_t(y)\in M_{a_{i+1}}$.

On the manifold $M_{[a_i, a_{i+1}]}$ there is a positive function $\alpha$ such that the flow along the scaled vector field $\alpha\cdot \nabla f$ takes any point on $M_{a_{i}}\backslash H$ to a point on $M_{a_{i+1}}\backslash H$ at the same time $t=1$ for all points on $M_{a_{i}}\backslash H$.

%\begin{center}
\begin{figure}
\includegraphics[width=0.45\textwidth]{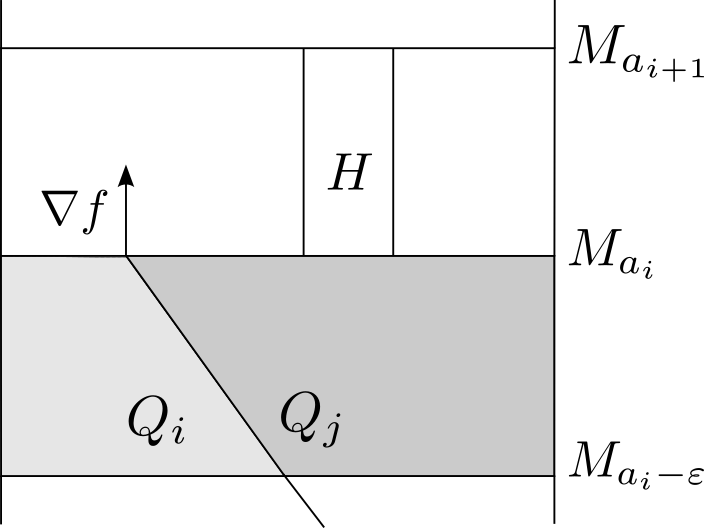}
\caption{The gradient field $\nabla f$, and the elements $Q_i$, $Q_j$ of a filling of $M_{[0, a_i]}$.}
\label{bcat1}
\end{figure}
%\end{center}
\begin{figure}
	\includegraphics[width=.45\textwidth]{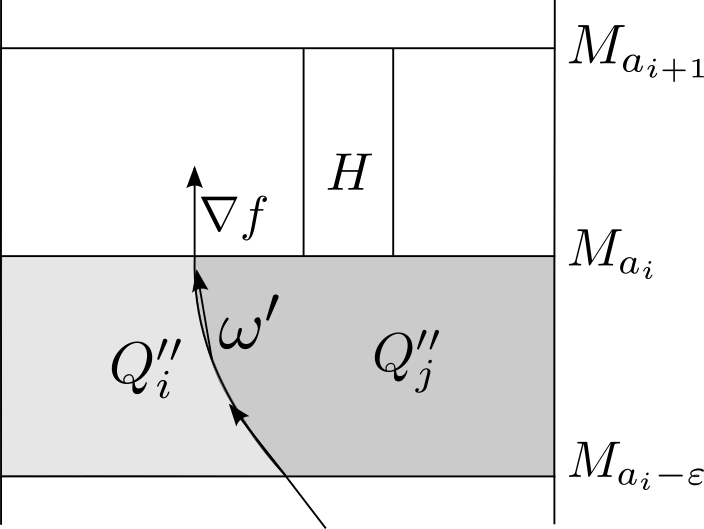}
	\caption{The gradient field $w'$ shown at 3 points with the modified filling.}
\label{bcat2}
\end{figure}

Since $Q_0,..., Q_\ell$ is a relative filling of $M_{[0, a_i]}$, the strata of the boundary $Q_j$ intersects each submanifold $M_{t}$ transversally for $t\in [a_i-\varepsilon, a_i)$ and all $j$ for a sufficiently small number $\varepsilon>0$. In particular, there is a gradient like vector field $w$ on $M_{[a_i-\varepsilon, a_i)}$ such that $w$ is tangent to each stratum of $\partial Q_j$ when restricted to $\partial Q_j$ for all $j$, see Fig.~\ref{bcat1}. 

Our aim is to extend the present relative filling $\{Q_i\}$ to a relative filling of $M_{[0, a_{i+1}]}$. To this end, we 
we will next modify the vector field $w$ and the relative filling of $M_{[0, a_i]}$ to a vector filed $w'$ and a relative filling $\{Q_{i}''\}$. 

 Let $\lambda$ be a smooth monotone function on $[0, a_{n+1}]$ such that $\lambda(t)= 0$ for all $t\in [0,a_i-\varepsilon]$, and $\lambda(t)= 1$ for all $t\in [a_i, a_{n+1}]$. Define a vector filed $w'$ on $M_{[a_{i}-\varepsilon, a_i]}$ by 
\[
w'(x)=	(1-\lambda (f(x)))w(x)+\lambda(f(x))\nabla f(x). 
\]
Then $w'$ is a vector field that agrees with $w$ near
$M_{a_i-\varepsilon}$, and agrees with $\nabla f$ near $M_{a_i}$. 
We now define a new relative filling of $M_{[0, a_i]}$ by sets $Q_0'', ..., Q_{\ell}''$, where $Q_j''$ is the union of $Q_j\cap M_{[0, a_i-\varepsilon)}$ and the points $y$ in $M_{[a_i-\varepsilon, a_i]}$ that flow along $-w'$ to  $Q_j\cap M_{a_i-\varepsilon}$, see Fig.~\ref{bcat2}. To show that $\{Q_i''\}$ is a relative filling, we observe that the submanifold $M_{a_i-\varepsilon, a_i}$ is diffeomorphic to $M_{a_i-\varepsilon}\times [0,1]$ under a diffeomorphism that takes any flow line of $w'$ in $M_{a_i-\varepsilon}$ to a segment of the form $\{x\}\times [0,1]$. The sets $Q_i''$ coincide with the elements $Q_i$ of a filling except over the part $M_{a_i-\varepsilon}\times [0,1]$. On the other hand, the sets $Q_i''\cap (M_{a_i-\varepsilon}\times [0,1])$ are cylinders $(Q_i''\cap M_{a_i-\varepsilon})\times [0,1]$. Therefore, the sets $Q_i''$ indeed form a relative filling. 

We observe that the filling $\{Q_i''\}$ of $M_{[0, a_i]}$ has the property that each non-horizontal  stratum of each element $Q_j''$ is vertically up near $M_{a_i}$.  

\begin{figure}
\includegraphics[width=.45\textwidth]{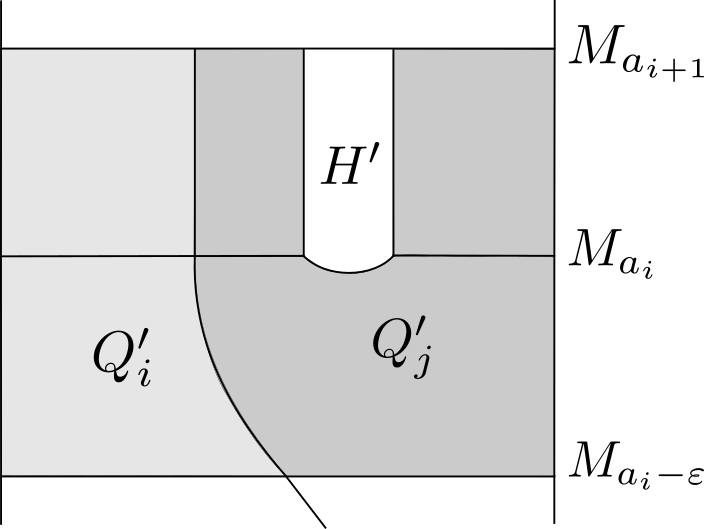}
\caption{The relative filling $\{Q_i'\}$.}
\label{bcat3}
\end{figure}

We next modify the vector field $w'$ to a vector field $v$ and extend it over $M_{[0, a_{i+1}]}$ which we will use to smooth corners of $H$, and extend the relative filling $\{Q_i''\}$ to a relative filling of $M_{[0, a_{i+1}]}$ that consists of extensions $Q_i'$ of $Q_i''$ and the smoothed version $H'$ of $H$. 

There is a smooth vector field  $v(x)$ on $M_{[0, a_{i+1}]}$ such that 
\begin{itemize} 
\item $v$ is zero over $M_{[0, a_i-\varepsilon]}$, 
\item $v$ is a multiple of $w'$ over $M_{[a_i-\varepsilon, a_i]}$,
\item $v$ is a non-zero multiple of $w'$ over $M_{[a_i-\varepsilon/2, a_i]}$, and
\item $v$ agrees with the scaled gradient vector field $\alpha\cdot \nabla f$ on $M_{[a_i, a_{i+1}]}\setminus H$.
\end{itemize}
In fact, we may choose $v(x)$ to be of the form $v(x)=\lambda_1(f(x))w'+\lambda_2(f(x))\nabla f$ for some smooth functions $\lambda_1$ and $\lambda_2$. Let 
\[
\rho\co M_{[0, a_{i}]}\times [0,1]\to M_{[0, a_{i+1}]}\] 
denote the flow along the vector field $v$ for $t\in[0,1]$. We note that $\rho_1$ takes $M_{a_i}\bs H$ diffeomorphically onto $M_{a_{i+1}}\bs H$, where $\rho_t(x)=\rho(x,t)$. 

Next we replace $H$ with a smooth submanifold $H'$ of $M_{[0, a_{i+1}]}$ with boundary that is smooth away from $M_{a_{i+1}}$. 
To this end we use the flow $\rho$ to identify a neighborhood of $H_{a_i}=H\cap M_{a_i}$ with $H_{a_i}\times [-\varepsilon/2, \varepsilon/2]$ by a diffeomorphism that takes a point $x$ in a neighborhood of $H_{a_i}$ with a pair $(y, t)$ such that $x=\rho_t(y)$. Let $\pi_{i}$ denote the projection of the product neighborhood onto the $i$-th factor where $i=1,2$. Given a function $\psi\le 0$ on $H_{a_i}$ with values in $(-\varepsilon/2, 0]$, let $H_\psi$ denote the union of $H$ and the points $x$ in the neighborhood $H_{a_i}\times [-\varepsilon/2, \varepsilon/2]$ above the graph of the function $\psi$, i.e., with points satisfying $\pi_2(x)\ge \psi(\pi_1(x))$. We may choose $\psi$ so that $H'=H_\psi$ is a submanifold of $M$ with corners only on $H'_{a_{i+1}}=H'\cap M_{a_{i+1}}$. 
 
Let $Q_i'$ denote the subset $\rho_1(Q_i'')\setminus H'$, see Fig.~\ref{bcat3}.  We claim that $H', Q_0', Q_1',\ldots, Q_\ell'$ is a relative filling of $M_{[0, a_{i+1}]}$. We note that the ordering of the elements of this relative filling is important and that $H'$ is the first element. Let $L$ be $M_{a_{i+1}}\times[0,1]$, and let $\tilde{M}$ stands for $L\cup M_{[0, a_{i+1}]}$ where the union is taken by identifying $M_{a_{i+1}}\times \{0\}\subset L$ with the boundary of $M_{[0, a_{i+1}]}$. We see that the sets $L, H', Q_0', Q_1',\ldots, Q_\ell'$ satisfy properties $1$ and $2$ for a relative filling of $\tilde{M}$ by construction. All that is left to show is that these sets satisfy property $3$ for all points $z\in \tilde{M}\bs\partial\tilde{M}$.

The sets $Q_0, Q_1, \ldots, Q_\ell$ define a filling  of the manifold $M_{a_i}$ by the sets $R_i=Q_i\cap M_{a_i}$. Let $H_0=H'\cap M_{a_i}$. Since $H'$ is a smooth compact submanifold of $M_{[0, a_{i+1}]}$ such that $\partial H'$ intersects the level $M_{a_i}$ transversely, we conclude that $H_0$ is a smooth compact submanifold of $M_{a_i}$. 
By Lemma~\ref{Emma}, the submanifold $H_0$ together with the sets $R_i'=\overline{R_i\setminus H_0}$ define a new filling of $M_{a_i}$; if $R_i'$ is empty, then we do not include it in the filling. 
 Let $H_1=H'\cap M_{a_{i+1}}$, and let $R_i''$ denote $\rho_1(R_i')$. Then $H_1$ together with $R_i''$  form a filling of $M_{a_{i+1}}$. 

 We shall now prove that property $3$ is satisfied for all $z\in \tilde{M}\bs\partial\tilde{M}$. We will consider four cases.
\begin{enumerate}
\item If $z\in M_{[0,a_{i+1}]}\bs H'$
then property $3$ is satisfied for $z$ since property $3$ is preserved under diffeomorphism, and the sets $Q'_i$ are images of the sets $Q_i''$ of a filling under the diffeomorphism $\rho_1$. 

\item If $z\in int(H')$ or $z\in int(L)$ then property $3$ is satisfied since the number of submanifolds of the filling containing that point is $1$.

\item Suppose $z$ is in the intersection $\partial H'\cap M_{(a_i,a_{i+1})}$. Let $z'$ be a point on $M_{a_i}$ on the gradient flow line passing through $z$, i.e., $\rho_{t_0}(z')=z$ for some $t_0>0$ Since the sets $H_0$ together with $R_i'$ form a filling of $M_{a_i}$, there is a special coordinate chart $D'$ of $z'$ as in Definition~\ref{Filling}. Let $D_\varepsilon\subset M_{[0, a_{i+1}]}$ be a coordinate chart  about $z$ given by 
\[
 D_\varepsilon = D'\times (-\varepsilon, \varepsilon)\subset M_{[0, a_{i+1}]},
\]
\[
(x, t) = \rho_{t_0+t}(x).  
\]
 We may choose $\varepsilon$ sufficiently small so that $D_\varepsilon$ is a subset in $H'$. Then $D_\varepsilon$ is a special coordinate chart about $z$ as in Definition~\ref{Filling}.

\item  Suppose that $z$ is in the intersection $H'_{a_{i+1}}=H'\cap M_{a_{i+1}}$. If $z$ is in the interior of $H'_{a_{i+1}}$, then the property $3$ is clearly satisfied. If $z$ is in the boundary of $H'_{a_{i+1}}$, then the property $3$ is satisfied since $\{H_1, R''_0, ... R''_\ell\}$ is a filling of $M_{a_{i+1}}$.

\item Suppose that $z$ is in the intersection $J=\partial H'\cap M_{[0,a_i]}$. If $z$ is in the interior of $J$, then the property $3$ is satisfied for $z$ since $\{Q_i''\}$ is a relative filling for $M_{[0, a_i]}$. If $z$ is in the boundary of $J$, then the property $3$ is satisfied for $z$ by the reason as in the case $3$. 

%, then property $3$ is satisfied. To see this, we note that $\partial H\cap M_{[0,a_i]}$ has a filling of codimension $1$ because $H'\cap M_{a_i}$ had a filling in $M_{a_i}$ and fillings are preserved under almost diffeomorphism \textcolor{red}{?}. So we see that the coordinate chart for $z$ will be the same coordinate chart for $(M_{a_i}\cap H')\times [0,1]$. Similarly, we see that property $3$ is satisfied if $z\in \partial H'\cap L$

\end{enumerate}
Thus,  the sets $H', Q_0',Q_1',\ldots Q_\ell'$ form a relative filling for $M_{[0,a_{i+1}]}$. Therefore as we pass from $M_{[0,a_i]}$ to $M_{[0,a_{i+1}]}$, the number of elements in a minimizing filling increases by $1$. Consequently, we have 
\[
\Bcat^\T(M_{[0,a_{i+1}]})+1\leq \Bcat^\T(M_{[0,a_i]})+2\leq \Critb(M_{[0,a_i]})+1\leq \Critb(M_{[0,a_{i+1}]}).\]

When $i=n$, the relative filling of $M_{[0, a_{i+1}]}$ is a filling of $M$. Consequently
\[
\Bcat^\T(M)+1\leq \Critb(M).
\]
which completes the proof.

\end{proof}

\section{The converse to the Lusternik-Schnirelmann inequality}\label{s:5}

In this section we will establish the converse to the Lusternik-Schnirelmann inequality. Namely, we will show that for a closed manifold $M$ there is an estimate: 
\[
\Bcat(M)+1\ge \Critb(M).
\]

Let $Q$ be a compact submanifold with corners of codimension $0$ in a smooth manifold $M$. A \emph{smooth collar neighborhood} of $Q$ in $M$ is a compact submanifold $Q'\subset M$ with smooth boundary such that there is an almost diffeomorphism $Q'\to Q\cup (\partial Q\times [0,1])$, where the union is taken by identifying $\partial Q\subset Q$ with $\partial Q\times \{0\}\subset \partial Q\times [0,1]$.  

 We recall that a vector field over a compact subset $Q$ of a smooth manifold is said to be \emph{smooth} if it is the restriction of a smooth vector field defined over a neighborhood of $Q$. 
 
Let $X$ be a topological submanifold of codimension $1$ of a Riemannian manifold $M$. Let $v$ be a smooth vector field over $X$ in $M$. Roughly speaking, the vector field $v$ is \emph{transverse} to $X$ at a point $x\in X$ if the projection of a neighborhood of $x$ in $X$ to a disc of dimension $\dim M-1$  transverse to $v(x)$ is a homeomorphism onto image. More precisely, 
 for each point $x\in X$,  let $N_x$ denote the unique geodesic in $M$ parametrized by $t\in (-\varepsilon, \varepsilon)$ such that at the time $t=0$ it passes through $x$, and its velocity at $t=0$ is $v$. If $\varepsilon>0$ is sufficiently small, then the geodesic $N_x$ exists for each $x\in X$. 
Suppose that for each point $x\in X$ there is a smooth disc $D_x$ in $M$ centered at $x$ such that $D_x$ is transverse to $N_x$ and 
$\{D_x\}$ is a continuous family of discs. Then for any point $x\in X$ we may identify $D_x\times N_x$ with a neighborhood of $x$ so that $D_x\times \{0\}$ is identified with $D_x\subset M$, and $0\times N_x$ is identified with $N_x\subset M$. Suppose that a neighborhood of $x$ in $X$ is the graph $\{(y, f_x(y))\}\subset D_x\times N_x$ of a continuous function $f_x\co D_x\to N_x$. Then we say that $v$ is \emph{transverse} to $X$ at $x$. If $v$ is transverse to $X$ at every point $x\in X$, then we say that $v$ is \emph{transverse} to $X$.  

Given a vector field $v$ transverse to $X$, let $\sigma\co X\to M$ be a continuous function that associates with each point $x$ a point in $N_x$. Then $\sigma(X)$ is homeomorphic to $X$. By the Cairns-Whitehead theorem, for every $\delta>0$ there is a continuous function $\sigma$ as above such that $\sigma(X)$ is a smooth submanifold of $M$ in the $\delta$-neighborhood of $X$. 

Similarly, let $M$ be a smooth manifold, and $Q$ be a set in a Singhof-Takens filling of $M$ by smooth manifolds with corners.  Then the Cairns-Whitehead smooth approximation of the boundary of $Q$ defines a smooth manifold (with smooth boundary) $Q'$ approximating $Q$. Lemma~\ref{lQ:15} shows that we may choose $Q'$ so that $Q\subset Q'$. 

\begin{remark}\label{rem:CWhi} In the Cairns-Whitehead construction above, if $\varepsilon>0$ is small enough, then different geodesic segments $N_x$  do not intersect. Let $\pi$ denote the projection of 
the neighborhood $\cup N_x$ of $X=\partial Q$ to $X$ defined by projecting each fiber $N_x$ to $x$. 
The smoothing $\sigma(X)$ is determined by the Riemannian metric on $M$, the value $\varepsilon$, the vector field $v$ as well as a smooth approximation of $\pi$. Since any two transverse vector fields over $X$ are homotopic, and any two smooth approximations of $\pi$ are isotopic, any two smooth approximations of $X$ are isotopic. In particular, if a Cairns-Whitehead smoothing of a manifold $Q$ with corners in a Singhof-Takens filling is diffeomorphic to a ball, then any other Cairns-Whitehead smoothing of $Q$ is also diffeomorphic to a ball. Thus, the notion of a smooth ball with corners is well-defined. 
\end{remark}

\begin{lemma}\label{lQ:15} Let $Q$ be a compact submanifold with corners of codimension $0$ in a smooth manifold $M$. Then $Q$ possesses a collar neighborhood. 
\end{lemma}
\begin{proof}  There is a smooth vector field $v$ over $\partial Q$ in $M$ transverse to $\partial Q$. Choose a Riemannian metric on $M$. For each point $x\in Q$,  let $\gamma_t(x)$ denote the geodesic parametrized by $t$ such that $\gamma_0(x)=x$, and $\dot\gamma_t(x)|_{t=0}=v(x)$. 
Let $exp$ denote the map $Q\to M$ that associates with each point $x$ the point $\gamma_1(x)$.    

If the vector field $v$ is sufficiently small,  then the image $X$ of $\partial Q$ under the exponential map $exp$ in the direction $v$ is a topological submanifold of $M$ homeomorphic to $\partial Q$. The vectors $v(x)$ translate over the geodesics $\gamma_t(x)$ to vectors $v_X(exp(x))$ so that $v_X$ is a vector field over $X$. In view of the transverse vector field $v_X$, by the Cairns-Whitehead theorem~\cite{Pu02}, there is a smooth approximation $Y$ of $X$. Then the manifold bounded by $Y$ that consists of $Q$ and parts $(x, \exp(x))$ of the shortest geodesics $\gamma_t(x)$  is a collar neighborhood of $Q$. 
\end{proof}

Let $M$ be a smooth manifold of dimension $n$ with boundary $\partial M$. A relative Singhof-Takens filling of $M$ by two subsets $Q_1$ and  $Q_2$ is said to be \emph{nice} if the set $Q_1$ contains $\partial M$, the pair $(Q_1, \partial M)$ is almost diffeomorphic to the pair $(\partial M\times [0,1], \partial M\times \{0\})$, and $Q_2$ is diffeomorphic to the disc $D^n$. We note that in this case both  manifolds $Q_1$ and $Q_2$ are smooth submanifolds with (smooth) boundary. 

Let now $M$ be a compact connected smooth manifold  of dimension $n$ such that $\partial M$ is the disjoint union of two manifolds $\partial_1 M$ and $\partial_2 M$.  Suppose that  $M=Q_3\cup Q_1\cup Q_2$ is a relative Singhof-Takens filling of $M$ such that $Q_1$ contains $\partial_1 M$, while $Q_2$ contains $\partial_2 M$. Suppose that  $(Q_i, \partial_iM)$ is almost diffeomorphic to $(\partial_iM\times [0,1], \partial_iM \times\{0\})$ for $i=1,2$, while $Q_3$ is diffeomorphic to a disc $D^n$. Then we say that the filling $\{Q_i\}$ is \emph{nice}.

\begin{remark}
In the original definition of a nice filling by Takens \cite{Ta68}, the boundary of $Q_3$ is required to be smooth.  We omit this requirement since the corners of a manifold can be smoothed. 
\end{remark}

To establish the converse of the Lusternik-Schnirelmann inequality we will rely on the Takens Theorem. 

\begin{theorem}[see Corollary 2.8 in \cite{Ta68}]\label{th:13}
Let $M$ be a compact manifold with a filtration  $\emptyset = M_0 \subset M_1\subset M_2\subset \ldots\subset M_k=M$ by compact submanifolds with boundary such that $M_i\subset \Int(M_{i+1})$ and $\partial M\subset M_k\setminus M_{k-1}$. Suppose that for each $i=1,..., k$, the manifold $M_i\setminus \Int(M_{i-1})$ with boundary $\partial M_{i-1}\sqcup \partial M_{i}$ admits a nice filling. Then $\Crit(M)\le k$. 
\end{theorem}

\begin{proposition}\label{p:Tak19} Under the hypotheses of Theorem~\ref{th:13}, we have $\Crit^\bullet(M)\le k$.
\end{proposition}
\begin{proof}  To prove Proposition~\ref{p:Tak19} it suffices to show that given a nice feeling $M=Q_3\cup Q_1\cup Q_2$ of a compact manifold $M$ with boundary components $\partial_1M\subset Q_1$ and $\partial_2M\subset Q_2$, there exists a  function that is 
\begin{itemize}
\item minimal, constant, and regular on $\partial_1 M$, \item maximal, constant, and regular on $\partial_2M$, and 
\item has one critical point in the interior of $M$. 
\end{itemize}

Furthermore, this critical point admits a cylindrical neighborhood diffeomorphic to a ball.   

The construction of the desired function is performed  in two steps. First, we choose a function $f$ on $M$  that is regular on the boundary, and such that $f|\partial_1M=-1$, $f|\partial_2M=1$, and $f|Q_3=0$. Furthermore,  we may choose the function  $f$ so that $f|(M\setminus Q_3)$ is smooth and has no singular points, and $f^{-1}(0)=(Q_1\cap Q_2)\cup Q_3$, see \cite[p.203]{Ta68}. Let $\Psi\co M\to M$ be a continuous map that takes $Q_3$ to a point $p\subset Q_3$ and that is a diffeomorphism of $M\setminus Q_3$ 
to $M\setminus \{p\}$. Define a new function $\bar f$ on $M$ by $f=\bar f\circ \Psi$. Then $\bar f$ is a continuous function that is smooth on $M\setminus \{p\}$. The critical point of $\bar{f}$ is a cone-like critical point. Indeed, the set $Q_1\cap Q_2\cap Q_3$ is a submanifold of $\partial Q_3$. In fact, it is the boundary of the submanifold $Q_1\cap Q_2$. Consequently, the critical point of $\bar{f}$ admits a neighborhood in $\bar{f}^{-1}(0)$ homeomorphic to the cone over $Q_1\cap Q_2\cap Q_3$. 
 In particular, the critical point of $\bar{f}$ admits a cylindrical ball neighborhood, see \cite{ST}.

Theorem 2.7 in \cite{Ta68} asserts that a function $\bar{f}$ can be modified in an arbitrarily small neighborhood $U_0$ of $p$ so that the modified function $\tilde f$ is smooth and has exactly one critical point. 

We note that the gradient flow curves of $\bar{f}$ are the same as those of $\tilde{f}$ except for those gradient flow curves that pass through the neighborhood $U_0$.  Therefore, every cylindrical neighborhood of the critical point $p$ of $\bar{f}$ whose interior contains $U_0$ is a cylindrical neighborhood of the critical point of $\tilde{f}$. 
\end{proof}

We are now in position to prove the converse of the Lusternik-Schnirelmann inequality.

\begin{theorem}\label{th:18upper}
For any closed manifold $M$, we have $\Bcat(M)+1\geq \Crit^\bullet(M)$.
\end{theorem}
\begin{proof}
Suppose that $\Bcat(M)=k-1$. In particular, the manifold $M$ admits a Singhof-Takens filling of $M=U_1\cup \cdots\cup U_k$  by smooth closed balls with corners.
Let $M_1$ be a smooth compact collar neighborhood of $U_1$, and, by induction for $i=2,..., k$, let $M_i$ be the union of $M_{i-1}$ and a smooth compact collar neighborhood of $U_i$ if $U_i$ is disjoint from $M_{i-1}$, and let $M_{i}$ be a smooth compact collar neighborhood of $U_i\cup M_{i-1}$ otherwise. Then $\{M_i\}$ forms a filtration of $M$ by compact submanifolds with (smooth) boundary. 

To begin with let us show that $M_1$ admits a nice filling by sets $Q_1$ and $Q_2$.  Put $Q_1=M_1\setminus \Int(U_1)$ and $Q_2=U_1$.  Since $M_1$ is a collar neighborhood of $U_1$, the pair $(Q_1,\partial M_{1})$ is almost diffeomorphic to the pair $(\partial M_{1}\times [0,1],\partial M_{1}\times\{0\})$.
Since $Q_2$ coincides with $U_1$ and $U_1$ is the first set in a Singhof-Takens filling $\{U_i\}$, the set $Q_2$ is diffeomorphic to a ball. Thus, the sets $Q_2$ and $Q_1$ form a nice filling of $M_1$.

\begin{figure}
\includegraphics[width=.7\linewidth]{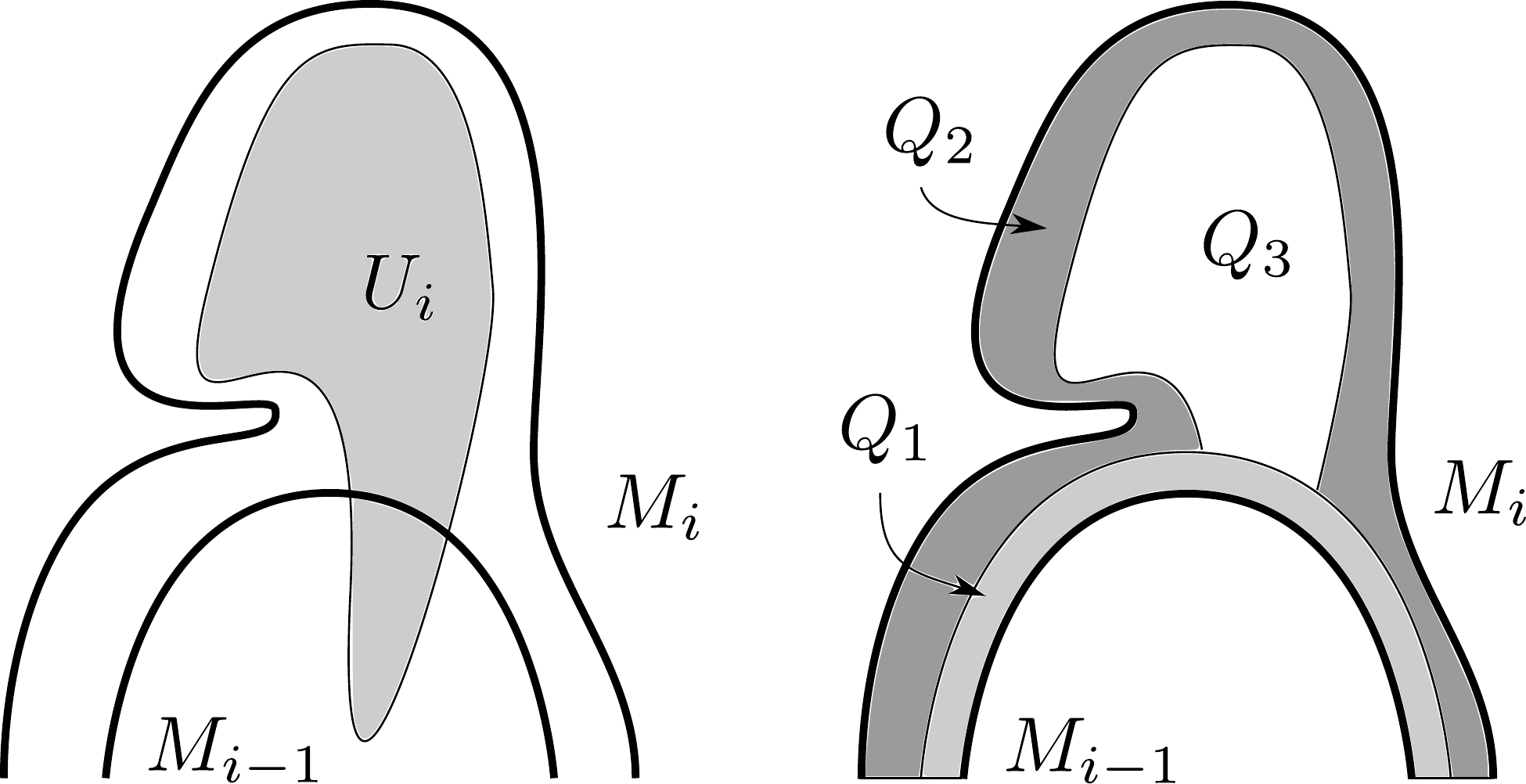}
\caption{The smooth compact collar neighborhood $M_i$, for $i>1$, of $U_i\cup M_{i-1}$ (on the left), and a nice filling of $M_i\setminus \Int(M_{i-1})$ on the right by $Q_1, Q_2$ and $Q_3$ provided that $U_i$ is not disjoint from $M_{i-1}$. The areas $Q_1$, $Q_2$ and $Q_3$  are light grey, dark grey and white respectively. }
\label{fig:1}
\end{figure}

For $i>1$ if $U_i$ is disjoint from $M_{i-1}$, then $M_i\setminus M_{i-1}$ admits a nice filling by a single smooth ball. Suppose now that $U_i$ is not disjoint from $M_{i-1}$.
Let us show that $M_i\setminus \Int(M_{i-1})$ admits a nice filling for $i>1$ by sets $Q_1, Q_2$ and $Q_3$, see Fig.~\ref{fig:1}. 
Let $Q_1$ be a collar neighborhood of $\partial M_{i-1}$ in $M_i\setminus \Int(M_{i-1})$. Since $M_{i-1}$ is a submanifold of $M$ with smooth boundary, the same is true for $Q_1$. We may assume that the intersection $Q_1\cap U_i$ is a neighborhood of $\partial M_{i-1}\cap U_i$ in $U_i\setminus \Int(M_{i-1})$ of the form $(\partial M_{i-1}\cap U_i)\times [0,1]$. We define $Q_2$ to be the complement in $M_i\setminus \Int(M_{i-1})$ to  $\Int(U_i\cup Q_1)$, and $Q_3$ to be the complement in $U_i$ to $\Int(Q_1\cup M_{i-1})$.

\iffalse
\begin{figure}
\begin{minipage}{.5\textwidth}
	\centering
	\includegraphics[width=.6\linewidth]{almost.png}
\end{minipage}%
\begin{minipage}{.5\textwidth}
	\centering
	\includegraphics[width=.6\linewidth]{nice.png}
\end{minipage}
\end{figure}
\fi 

%\includegraphics[scale=.25]{almost nice.png}

%\includegraphics[scale=.25]{nice.png}

 As the sets $Q_2$ and $Q_3$ may have high order corners, the sets $Q_1, Q_2$ and $Q_3$ do not form a filling of the manifold $M_i\setminus \Int M_{i-1}$. However, 
 the three sets $Q_1, Q_2$ and $Q_3$ clearly cover the manifold  $M_i\setminus \Int M_{i-1}$,  and their interiors are disjoint.  
Since $Q_1$ is a collar neighborhood of the smooth boundary component $\partial M_{i-1}$ of the smooth submanifold $M_i\setminus \Int(M_{i-1})$, the pair $(Q_1,\partial M_{i-1})$ is almost diffeomorphic to the pair $(\partial M_{i-1}\times [0,1],\partial M_{i-1}\times\{0\})$. Since $Q_1\cap U_i$ is of the form $(\partial M_{i-1}\cap U_i)\times [0,1]$ and $U_i$ is a smooth ball with corners, it follows that $Q_3=U_i\setminus \Int(Q_1\cup M_{i-1})$ is a smooth ball with corners as well. Let us show that $Q_2$ is almost diffeomorphic to $\partial M_i\times [0,1]$. 
We observe that $Q_2=(M_i\setminus \Int M_{i-1})\setminus \Int(Q_1\cup U_i)$ can be written as 
\[
M_i \setminus\  \Int(M_{i-1} \cup (\partial M_{i-1}\times [0,1])\cup U_i)
\]
where we identify $Q_1$ with $ \partial M_{i-1} \times [0,1]$. Consequently, the manifold with corners $Q_2$ is almost 
diffeomorphic to $M_i\setminus \Int(M_{i-1}\cup U_i)$. On the other hand, by definition, the manifold $M_i$ is a collar neighborhood of $M_{i-1}\cup U_i$. Therefore, up to almost diffeomorphism, the complement to $\Int(M_{i-1}\cup U_i)$ in $M_i$ is the collar $\partial M_i\times [0,1]$, as required. 

Let $Q_3'$ be a collar neighborhood of $Q_3$ in $M$, see Lemma~\ref{lQ:15}, such that $Q_3'$ is a submanifold in the interior of $M_{i}$ and $Q_3'$ is disjoint from $M_{i-1}$. Then $Q_3'$ is diffeomorphic to a ball. We now modify $Q_1$ by replacing it with a new set $Q_1'$ obtained from a collar neighborhood $\mathcal{U}(Q_1)$ of $Q_1$ by removing the interior of $Q_3'\cap \mathcal{U}(Q_1)$. In this construction we may choose the collar neighborhood $\mathcal{U}(Q_1)$ of $Q_1$ so that its boundary intersects the boundary of $Q_3'$ transversally. Finally, we redefine $Q_2'$ to be the complement in $M_i\setminus \Int M_{i-1}$ to the interior of the union of $Q_1'$ and $Q_3'$. The resulting sets form a nice filling of $M_i\setminus \Int M_{i-i}$.

Thus, indeed, for every $i$, the manifold $M_i\setminus \Int(M_{i-1})$ admits a nice filling. 

Therefore, by Proposition~\ref{p:Tak19}, the manifold $M$ admits a function with at most $k$ critical points. 
\end{proof}

\end{document}